\documentclass[11pt, reqno]{amsart}
\usepackage{amssymb}
\usepackage{amsmath}
\usepackage{amsthm}
\usepackage{mathabx}
\usepackage{braket}
\usepackage{pdfpages}
\usepackage{graphicx}
\usepackage{caption}
\usepackage{subcaption}
\usepackage{mathrsfs} 
\usepackage{tikz-cd} 
\usepackage{bm}
\usepackage{enumitem}
\usepackage{mathtools}
\usepackage{pgfplots}
\usepackage{import}
\usepackage{xifthen}
\usepackage{transparent}
\usepackage{mathtools}

\pgfplotsset{compat=1.14}

\usepackage{tikz}
\usetikzlibrary{automata, arrows.meta, positioning}

\usepackage[utf8]{inputenc}
\usepackage[english]{babel}

\usepackage{hyperref}
\usepackage{cleveref}

\usepackage{lmodern,babel,adjustbox,booktabs,multirow}

\crefformat{footnote}{#2\footnotemark[#1]#3}

\numberwithin{equation}{section}
\theoremstyle{plain}
\newtheorem{theorem}{Theorem}[section]

\newtheorem{prop}[theorem]{Proposition}
\newtheorem{conj}[theorem]{Conjecture}
\newtheorem{lem}[theorem]{Lemma}
\newtheorem{cor}[theorem]{Corollary}

\newtheorem*{question*}{Question}

\theoremstyle{plain}

\theoremstyle{definition}

\newtheorem{rmk}[theorem]{Remark}

\newtheorem*{remark}{Remark}

\newcommand{\R}{\mathbb{R}}
\newcommand{\C}{\mathbb{C}}

\newcommand{\Z}{\mathbb{Z}}
\newcommand{\Hyp}{\mathbb{H}}

\newcommand{\pcf}{post-critically finite }

\DeclareMathOperator{\sgn}{sgn}

\DeclareMathOperator{\Homeo}{Homeo}

\DeclareMathOperator{\Per}{Per}

\DeclareMathOperator{\PSL}{PSL}

\DeclareMathOperator{\Conf}{Conf}

\DeclareMathOperator{\Int}{Int}
\DeclareMathOperator{\chull}{Cvx\, Hull}

\DeclareMathOperator{\core}{core}

\numberwithin{figure}{section}

\begin{document}
\title[Quasiconformal non-equivalence]{On quasiconformal non-equivalence of gasket Julia sets and limit sets}
\begin{author}[Y.~Luo]{Yusheng Luo}
\address{Department of Mathematics, Cornell University, 212 Garden Ave, Ithaca, NY 14853, USA}
\email{yusheng.s.luo@gmail.com}
\thanks{Y.L. was partially supported by NSF Grant DMS-2349929.}
\end{author}
\begin{author}[Y.~Zhang]{Yongquan Zhang}
\address{Institute for Mathematical Sciences, Stony Brook University, 100 Nicolls Rd, Stony Brook, NY 11794-3660, USA}
\email{yqzhangmath@gmail.com}
\end{author}

\begin{abstract}
This paper studies quasiconformal non-equivalence of Julia sets and limit sets.
We proved that any Julia set is quasiconformally different from the Apollonian gasket.
We also proved that any Julia set of a quadratic rational map is quasiconformally different from the gasket limit set of a geometrically finite Kleinian group.
\end{abstract}

\maketitle

\setcounter{tocdepth}{1}
\tableofcontents

\section{Introduction}
It is a central question in quasiconformal geometry to classify fractal sets up to quasiconformal homeomorphisms.
For fractal sets that emerge in conformal dynamics, ample evidence suggests that it is possible to quasiconformally distinguish Julia sets and limit sets.
It is summarized as the following conjecture in \cite{LLMM19}.

\begin{conj}[\cite{LLMM19}]\label{conj:gqh}
Let $J$ be the Julia set of a rational map and $\Lambda$ be the limit set of a Kleinian group.
Suppose that $J$ and $\Lambda$ are connected, and not homeomorphic to a circle or a sphere.
Then $J$ is not quasiconformally homeomorphic to $\Lambda$.
\end{conj}

In this paper, we will study this quasiconformal non-equivalence phenomenon for some classes of Julia set and limit set.
Our first result is
\begin{theorem}\label{thm:thmA}
No Julia set of a rational map is quasiconformally homeomorphic to the Apollonian gasket.
\end{theorem}

\begin{figure}
    \centering
    \begin{subfigure}[c]{0.45\textwidth}
		\centering
		\includegraphics[width=\textwidth]{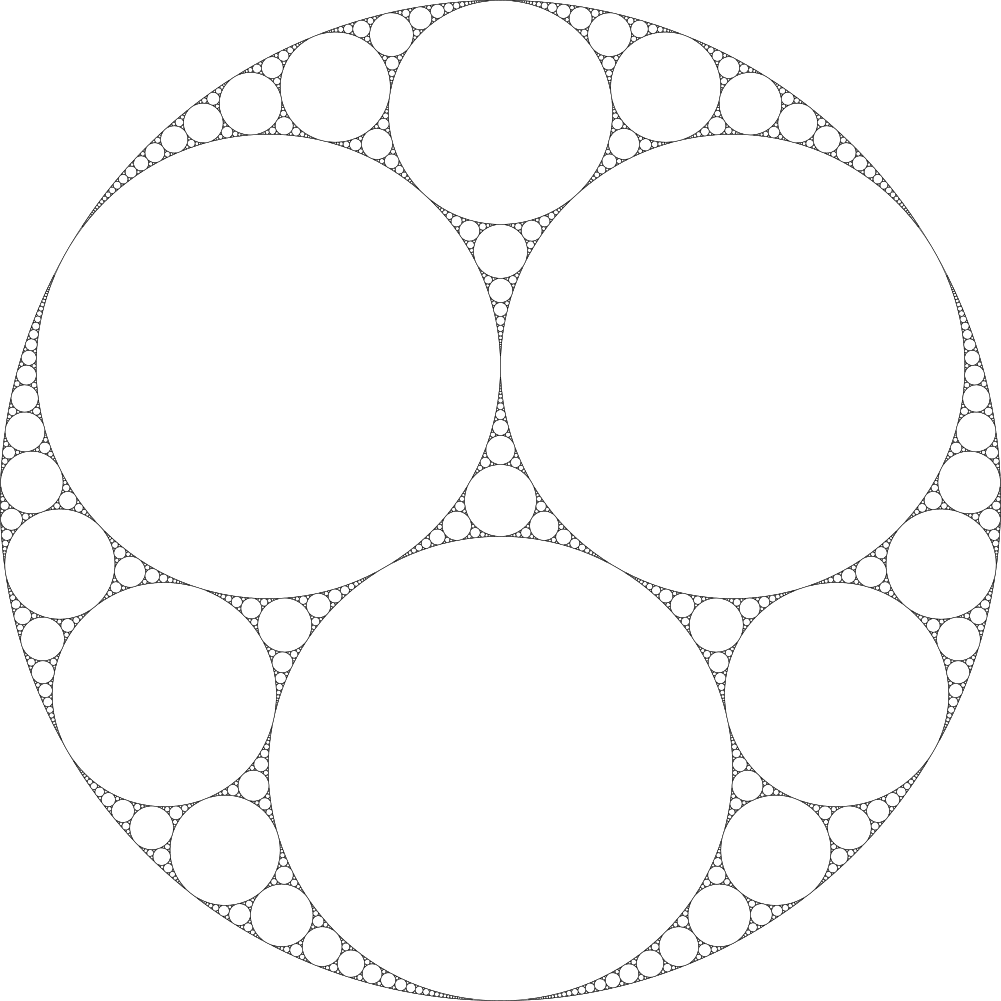}
	\end{subfigure}
	\begin{subfigure}[c]{0.45\textwidth}
		\centering
		\includegraphics[width=\textwidth]{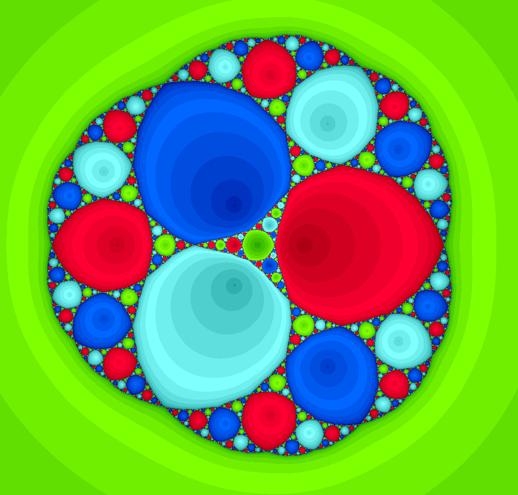}
	\end{subfigure}
    \caption{The Apollonian gasket on the left, and the Julia set that is homeomorphic to an Apollonian gasket on the right. The two sets are not quasiconformally homeomorphic as Fatou components touch at an angle on the right.}
    \label{fig:AG}
\end{figure}

\begin{remark}
We remark that there exist Julia sets that are homeomorphic to the Apollonian gasket (see Figure \ref{fig:AG}). The Apollonian gasket is the limit set of the Apollonian group (\cite{GLMWY05}). In fact, for the limit set of any geometrically finite Kleinian group that is homeomorphic to the Apollonian gasket, it is quasiconformally homeomorphic to the Apollonian gasket (see \cite{McM90} or \S \ref{sec:KG}). Therefore, Theorem \ref{thm:thmA} can be restated as no Julia set is quasiconformally homeomorphic to an Apollonian gasket limit set of a geometrically finite Kleinian group.
\end{remark}


More generally, we define a {\em gasket} $K$ as a closed subset of $\hat\C$ so that
\begin{enumerate}
    \item each complementary component is a Jordan domain;
    \item any two complementary components touch at most at 1 point;
    \item no three complementary components have a common boundary point;
    \item the {\em contact graph} (or the {\em nerve}), obtained by assigning a vertex to complementary component and an edge if two touch, is connected.
\end{enumerate}

Apollonian gasket provides one example of gaskets.
Gaskets also arise naturally as Julia set of rational maps and limit sets of Kleinian groups (see \S \ref{sec:KG}) and appear frequently in the literature.
For example, any {\em polyhedral circle packing} defined in \cite{KN19} gives a gasket limit set. 
Many rational maps have gasket Julia set that are homeomorphic to gasket limit sets (see \cite{LLM22, LLM22b}).
Our second result shows that


\begin{theorem}\label{thm:thmB}
No Julia set of a quadratic rational map is quasiconformally homeomorphic to a gasket limit set of a geometrically finite Kleinian group.
\end{theorem}

For higher degree rational maps, the combinatorics becomes more complicated.
However, we believe it is possible to generalize our method to higher degrees and prove Conjecture \ref{conj:gqh} in the gasket case.

\subsection{Historical Background}
Our study of quasiconformal non-equivalence between Julia sets and limit sets is motivated by the study of rigidity of quasisymmetries of a Sierpi\'nski carpet in \cite{BKM09, BM13, Mer14, MLM16}.
Using the rigidity, it is proved in \cite[Corollary 1.2]{MLM16} that a Sierpi\'nski carpet Julia set of a \pcf rational map has a finite quasisymmetry group.
Since the quasisymmetry group is a quasiconformal invariant, and the limit set has a large conformal symmetry group, we have immediately that (cf. \cite[Corollary 1.3]{Mer14})
\begin{theorem}{\cite[Corollary 1.3]{MLM16}}
No Julia set of a \pcf rational map is quasiconformally homeomorphic to a Sierpi\'nski carpet limit set of a Kleinian group.
\end{theorem}
This result is generalized in \cite{QYZ19} to {\em semi-hyperbolic} rational maps.

On the other hand, using David surgery, it is proved in \cite{LLMM19} that there exists some rational map $f$ so that
\begin{itemize}
    \item the Julia set $\mathcal{J}$ is homeomorphic to the Apollonian gasket; and
    \item the quasisymmetry group $\mathcal{QS}(\mathcal{J})$ equals the homeomorphism group $\Homeo(\mathcal{J})$.
\end{itemize}
The homeomorphism $\Phi: \mathcal{J} \longrightarrow \Lambda$ between the Julia set $\mathcal{J}$ and the Apollonian gasket $\Lambda$ is {\bf not} quasiconformal.
Yet, it induces an isomorphism
$$
\Phi_*: \mathcal{QS}(\mathcal{J}) \longrightarrow \mathcal{QS}(\Lambda) = \Conf(\Lambda).
$$
Thus, in this case, the quasisymmetry groups do not distinguish the Julia set $\mathcal{J}$ and the Apollonian gasket $\Lambda$.
The same phenomenon can also occur for other general gaskets (see \cite{LLMM19} for more details).

Although Sierpi\'nski carpet and gasket both arise as limit set of an acylindrical hyperbolic 3-manifold \cite{McM90}, 
the two cases are quite different.
We use the combinatorics of the gaskets to study the quasiconformal non-equivalence problem.

\subsection{Strategy and techniques}
In \S \ref{sec:KG}, we recall a characterization of finitely generated Kleinian groups with gasket limit sets (see Theorem \ref{thm:GFAC}).
If the Kleinian group is geometrically finite, then the corresponding hyperbolic 3-manifold $M = \Hyp^3/\Gamma$ is {\em acylindrical} in the sense of Thurston.
It follows that $\Gamma$ is quasiconformally conjugate to a Kleinian group with totally geodesic convex hull boundary.
Thus $\Lambda$ is quasiconformally homeomorphic to an infinite circle packing (see Corollary~\ref{cor:GFAC}).

Note that the degree of tangency at the intersection point of two complementary components of $\Lambda$ is a quasiconformal invariant.
Thus, a necessary condition for a gasket Julia set $\mathcal{J}$ to be quasiconformally homeomorphic to $\Lambda$ is that 
\begin{itemize}
    \item the boundary of each Fatou component contains no cusps;
    \item two Fatou components are tangent to each other if they touch.
\end{itemize}
We shall call such a Julia set $\mathcal{J}$ a {\em fat gasket} (see Figure \ref{fig:FG}).

\begin{figure}
    \centering
    \includegraphics[width=0.8\textwidth]{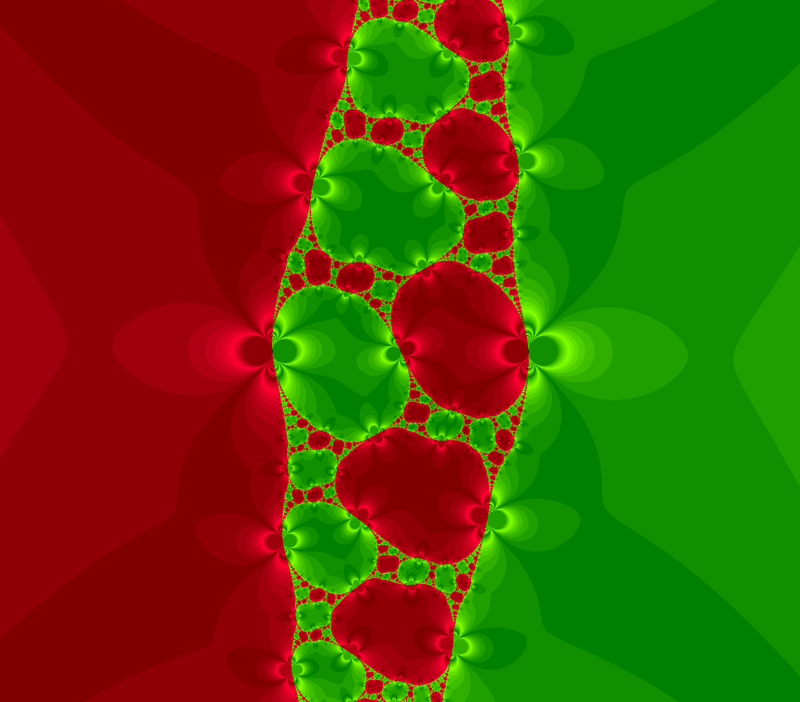}
    \caption{An example of a fat gasket Julia set. The Fatou graph is bipartite as one can see by the coloring of Fatou set.}
    \label{fig:FG}
\end{figure}

Let $\mathcal{G}$ be the contact graph of a fat gasket Julia set $\mathcal{J}$ of $f$, which we shall also call the {\em Fatou graph} of $f$.
We show that the local quasiconformal structure puts a global constraint on $\mathcal{G}$.
In particular, we prove that
\begin{theorem}\label{thm:thmC}
The Fatou graph $\mathcal{G}$ of a fat gasket Julia set is bipartite.
\end{theorem}
We remark that Theorem \ref{thm:thmC} allows one to show that many gasket Julia sets are quasiconformally different from gasket limit sets. 
For example, Theorem \ref{thm:thmA} follows immediately from it as the contact graph of the Apollonian gasket contains a cycle of length 3, so it is not bipartite.

On the other hand, there are plenty of rational maps with fat gasket Julia set and we use Thurton theory of rational maps to obtain a characterization (see Theorem \ref{thm:nto}).
In particular, for degree 2, we have
\begin{theorem}\label{thm:thmD}
Let $f$ be a quadratic rational map with a fat gasket Julia set. Then either
\begin{enumerate}[label=\normalfont{(\arabic*)}]
    \item the Fatou graph is a tree; or
    \item the Fatou graph is not a tree and $f$ is a root of a captured type hyperbolic component with an attracting cycle of period 2.
\end{enumerate}
Moreover, 
\begin{itemize}
    \item in Case (1), if $f$ is geometrically finite, then $f$ is a mating of the fat basilica with a Misiurewicz polynomial;
    \item in Case (2), any root of a captured type hyperbolic component with an attracting cycle of period 2 has fat gasket Julia set.
\end{itemize}
\end{theorem}
\begin{remark}
We recall that a hyperbolic component $\mathcal{H}$ of a quadratic rational map is called captured type if only one critical point is in the immediate basin of a periodic point.
A rational map $f$ is called a {\em root} of $\mathcal{H}$ if $f \in \partial \mathcal{H}$ and the dynamics on the Julia set $\mathcal{J}(f)$ is topologically conjugate to that of $g \in \mathcal{H}$.

We also recall that a quadratic polynomial $f(z) = z^2+c$ is called a Misiurewicz polynomial if the critical point $0$ is strictly pre-periodic.
The Julia set of a Misiurewicz polynomial is a dendrite. 
We remark that not every Misiurewicz polynomial $f$ is matable with the fat basilica.
In fact, $f$ is matable with the fat basilica if and only if $f$ is not in the 1/2-limb of the Mandelbrot set \cite{Tan92}.
We also remark even if $f$ is matable with the fat basilica, the Fatou graph of their mating may not be a tree.
In this case, by Theorem \ref{thm:thmD}, the mating is not a fat gasket Julia set.
\end{remark}

To prove Theorem \ref{thm:thmB}, we first note that the contact graph of a gasket limit set of a geometrically finite Kleinian group is not a tree (see Theorem \ref{thm:GFAC}).
Thus, we can assume the Fatou graph $\mathcal{G}$ is not a tree.
We then show the homeomorphism group $\Homeo(\mathcal{G})$ is very restrictive (see \S \ref{sec:QNE}).
This allows us to distinguish $\mathcal{G}$ from the contact graph of a gasket limit set.

\subsection{Notes and discussions}
The Apollonian gasket or the Apollonian circle packing and its arithmetic, geometric, and dynamical properties have been extensively studied in the literature (for a non-exhaustive list, see e.g.\ \cite{GLMWY03, GLMWY05, BF11, KO11, OS12, BK14, Zha22}).

More recently, there have been many new and exciting developments in the study of other gasket or circle packings coming from the limit set of Kleinian groups \cite{KN19, KK21, BKK22, LLM22}.

It is shown in \cite{LLM22} that many of these gaskets are homeomorphic to Julia sets of rational maps.
In the same paper, many other pairs of homeomorphic Julia set and limit set are constructed. 
It would be interesting to study the quasiconformal non-equivalence problem for such Julia sets and limit sets.

\subsection{Structure of the paper}
We study Kleinian groups with gasket limit sets in \S \ref{sec:KG}.
The induced dynamics on the Fatou graph of a rational map with a fat gasket Julia set is studied in \S \ref{sec:ifg}, where Theorem \ref{thm:thmC} is proved.
The realization of rational maps with fat gasket Julia sets are studied in \S \ref{sec:rfg}. In particular, Theorem \ref{thm:thmD} is proved there.
Finally, the combinatorics of Fatou graphs for quadratic fat gasket Julia sets are studied in \S \ref{sec:CT}, and Theorem \ref{thm:thmB} is proved in \S \ref{sec:QNE}.

\subsection*{Acknowledgement}
The authors thank Mikhail Lyubich and Sabyachi Mukherjee for bringing up the question and useful discussions.
The authors also acknowledge that the outline of the proof of Theorem \ref{thm:thmA} is suggested and communicated to us by Curt McMullen, and we thank him for many valuable suggestions. Y.Z. acknowledges the support of Max Planck Institute for Mathematics, where part of the research was done.

\section{Kleinian groups with gasket limit set}\label{sec:KG}
In this section, we recall some results on geometrically finite Kleinian groups, especially those with gasket limit sets. We refer to \cite[Appx.~B]{LZ23} for details.
Throughout the section, let $\Gamma\subset\PSL(2,\mathbb{C})$ be a Kleinian group, $\Lambda\subset\hat{\mathbb{C}}$ its limit set, and $\Omega=\hat{\mathbb{C}}-\Lambda$ its domain of discontinuity.
Up to a finite-index subgroup, we may assume $\Gamma$ is torsion-free.
We also denote by $M=\Gamma\backslash\mathbb{H}^3$ the corresponding hyperbolic three-manifold.

\subsection{Geometric finiteness and acylindricity}\label{subsec:acy}
Recall that the \emph{convex core} of $M$ is given by
$$\core(M):=\Gamma\backslash\chull(\Lambda),$$
where $\chull(\Lambda)$ denotes the convex hull of $\Lambda$ in $\mathbb{H}^3$.
We say $M$ (and the corresponding Kleinian group $\Gamma$) is \emph{geometrically finite} if the unit neighborhood of $\core(M)$ has finite volume.

Let $(N,P)$ be a \emph{pared manifold}, where $N$ is a compact oriented 3-manifold with boundary and $P\subset\partial N$ is a submanifold consisting of incompressible tori and annuli.
See \cite{hyperbolization1} for a precise definition in arbitrary dimension.

Set $\partial_0N=\partial N-P$. We say $(N,P)$ is \emph{acylindrical} if $\partial_0N$ is incompressible, and every cylinder
$$f:(S^1\times[0,1],S^1\times\{0,1\})\to(N,\partial_0N)$$
whose boundary components $f(S^1\times\{0\})$ and $f(S^1\times\{1\})$ are essential curves of $\partial_0N$ can be homotoped rel boundary into $\partial N$.

For a geometrically finite $M$, let $\core_\epsilon(M)$ be the convex core of $M$ minus $\epsilon$-thin cuspidal neighborhoods for all cusps. Here $\epsilon$ is chosen small enough, say smaller than the Margulis constant in dimension 3. Let $P\subset\partial\core_\epsilon(M)$ be the union of boundaries of all cuspidal neighborhoods. Then $(\core_\epsilon(M),P)$ is a pared manifold, and we say $M$ (and the corresponding Kleinian group $\Gamma$) is acylindrical if $(\core_\epsilon(M),P)$ is.

One can recognize acylindricity from the limit set in the geometrically finite case. The following characterization is well-known; see for example \cite[Prop.~B.2]{LZ23}.
\begin{prop}[Characterization of geom.\ finite acy.\ Kleinian groups]\label{prop:acy_geom_fin}
Suppose $\Gamma$ is nonelementary and geometrically finite of infinite volume. Then $\Gamma$ is acylindrical if and only if any connected component of the domain of discontinuity $\Omega$ is a Jordan domain, and the closures of any pair of connected components share at most one point. Moreover, any common point of the closures of two connected components is a parabolic fixed point, and any rank-1 parabolic fixed point arises this way.
\end{prop}

One ingredient in the proof is the following lemma, essentially from \cite[Theorem~3]{Mas74}.
\begin{lem}\label{lem:pfp}
Let $\Gamma$ be any nonelementary Kleinian group, and $\Omega_1$, $\Omega_2$ two connected components of its domain of discontinuity $\Omega$. Let $\Gamma_i$ be the stabilizer of $\Omega_i$ in $\Gamma$, and assume $X_i:=\Gamma_i\backslash\Omega_i$ is a Riemann surface of finite type. Suppose $\overline{\Omega_1}\cap\overline{\Omega_2}$ consists of one point $p$. Then
\begin{enumerate}[label=\normalfont{(\arabic*)}]
    \item $p$ is a parabolic fixed point;
    \item Let $\sigma_i$ be a curve on $X_i$ that is not null-homotopic in $\overline{M}:=\Gamma\backslash\{\mathbb{H}^3\cup\Omega\}$. Then $\sigma_1$ and $\sigma_2$ are not homotopic in $\overline{M}$.
\end{enumerate}
\end{lem}

A \emph{quasiconformal deformation} of $\Gamma$ is a discrete and faithful representation $\xi:\Gamma\to\PSL(2,\mathbb{C})$ that preserves parabolics, induced by a quasiconformal map $f:\hat{\mathbb{C}}\to\hat{\mathbb{C}}$ (i.e.~$\xi(\gamma)=f\circ\gamma\circ f^{-1}$ for all $\gamma\in\Gamma$). In particular, the limit sets of $\Gamma$ and $\xi(\Gamma)$ are quasiconformally homeomorphic via $f$.

The \emph{quasiconformal deformation space} of $\Gamma$ is defined by
$$\mathcal{QC}(\Gamma):=\{\xi:\Gamma\to\PSL(2,\mathbb{C})\text{ is a quasiconformal deformation}\}/\sim$$
where $\xi\sim\xi'$ if they are conjugate by a M\"obius transformation. This space is naturally identified with the quasi-isometric deformation space of the corresponding hyperbolic manifold $M$ (see e.g.\ \cite{Sul81}).

When $\Gamma$ is acylindrical, there exists a unique $\Gamma'\in\mathcal{QC}(\Gamma)$ so that the convex core of $M'=\Gamma'\backslash\mathbb{H}^3$ has totally geodesic boundary \cite[Corollary 4.3]{McM90}. Thus the limit set of any geometrically finite and acylindrical Kleinian group is quasiconformally homeomorphic to a circle packing.

\subsection{Gasket limit sets}
The following characterization of Kleinian groups with gasket limit sets was proved in \cite{LZ23}.
\begin{theorem}\label{thm:GFAC}
Suppose the limit set $\Lambda$ of a finitely generated Kleinian group $\Gamma$ is a gasket, and let $\mathcal{G}$ be its contact graph. Then the corresponding hyperbolic 3-manifold $M$ is homeomorphic to the interior of a compression body $N$. If furthermore $M$ is geometrically finite, then $\mathcal{G}$ is not a tree, $M$ is acylindrical, and the compression body $N$ has empty or only toroidal interior boundary components.
\end{theorem}
Here, a compact orientable irreducible manifold with boundary $(N,\partial N)$ is a \emph{compression body} if the inclusion of one boundary component $\partial_eN$ induces a surjection on $\pi_1$; we refer to this component as the \emph{exterior} boundary. All other boundary components of $N$ are incompressible; we refer to them as \emph{interior} boundary components.
In fact, when $\Gamma$ contains no rank-2 parabolic fixed points, the compression body $N$ has no interior boundary component, so it is a handlebody.
On the other hand, when $\mathcal{G}$ is a tree, $M$ is homeomorphic to $S\times\mathbb{R}$ for some surface $S$ of finite type, and has one geometrically infinite end.

The following result then follows from this characterization and discussion in the previous subsection.
\begin{cor}\label{cor:GFAC}
Let $\Lambda$ be a gasket limit set of a geometrically finite Kleinian group $\Gamma$.
Then $\Lambda$ is quasiconformally homeomorphic to an infinite circle packing.
\end{cor}
In particular, to prove our main result Theorem~\ref{thm:thmB}, we may restrict our attention to Kleinian groups with circle packing limit sets.

\section{Induced dynamics on the Fatou graph}\label{sec:ifg}
Recall that a gasket Julia set $\mathcal{J}$ is a fat gasket if 
\begin{itemize}
    \item the boundary of each Fatou component contains no cusps;
    \item two Fatou components are tangent to each other if they touch.
\end{itemize}
In this section, we shall prove that dynamics of a rational map with a fat gasket Julia set is restricted.
\begin{theorem}\label{thm:dfg}
Let $f$ be a rational map with a fat gasket Julia set.
Let $\mathcal{G}$ be the Fatou graph for $f$.
Then $f$ induces a simplicial map
$$
f_*: \mathcal{G} \longrightarrow \mathcal{G},
$$
and there exists a unique fixed edge $E_0 \subseteq \mathcal{G}$ of $f_*$ so that every edge is eventually mapped to $E_0$.
\end{theorem}
Here we recall that a map is called {\em simplicial} if it maps an edge to an edge.
It follows that the two boundary points $\partial E_0$ are either both fixed, or form a periodic cycle of period 2.
Thus, Theorem \ref{thm:thmC} follows immediately from Theorem \ref{thm:dfg}.
\begin{proof}[Proof of Theorem \ref{thm:thmC}]
After passing $f_*$ to the second iterate if necessary, we may assume that the boundary points $\partial E_0 = \{x, y\}$ are both fixed.
Then we can divide vertex set into two groups $U_x, U_y$ depending on whether the vertex is eventually mapped to $x$ or $y$.
This division gives the bipartite structure of the graph $\mathcal{G}$.
\end{proof}

The proof of Theorem \ref{thm:dfg} consists of several lemmas.
\begin{lem}\label{lem:njc}
Let $f$ be a rational map with a fat gasket Julia set.
Then no critical point is on the boundary of a Fatou component.
\end{lem}
\begin{proof}
Suppose for contradiction that $f$ has a critical point $c$ on the boundary of a Fatou component $U$.
Let $V = f(U)$.
Then $f(c) \in \partial V$.
Let $e > 1$ be the multiplicity of the critical point $c$, and $\mu$ be the number of Fatou components attached at $f(c)$.
Then there are $e\mu$ Fatou components attached at $c$.
Since at most $2$ Fatou components are attached to $c$, $e=2$ and $\mu = 1$.
Since the two Fatou components touch tangentially at $c$ and $f$ behaves like $z^2$ near $c$, the boundary of $V$ has a cusp at $f(c)$.
This is a contradiction, and the lemma follows.
\end{proof}

As an immediate corollary, we have
\begin{cor}\label{cor:is}
Let $f$ be a rational map with a fat gasket Julia set.
Then $f$ induces a simplicial map
$$
f_*: \mathcal{G} \longrightarrow \mathcal{G}.
$$
Moreover, every edge of $\mathcal{G}$ is pre-periodic.
\end{cor}
\begin{proof}
Since vertices in $\mathcal{G}$ are in bijective correspondence with Fatou components of $f$, $f$ naturally induces a map $f_*$ on the vertex set of $\mathcal{G}$.
By Lemma \ref{lem:njc}, there are no critical points on the boundary of a Fatou component. 
Thus, if $v,w$ are adjacent in $\mathcal{G}$, then $f_*(v) \neq f_*(w)$.
Hence, $f_*(v)$ and $f_*(w)$ are adjacent.
Therefore, we can extend the map $f_*$ to $\mathcal{G}$ by sending an edge $E = [v,w]$ to the edge $[f_*(v), f_*(w)]$.
By construction, the induced map is simplicial.

Since each vertex is pre-periodic, by construction, each edge is also pre-periodic.
\end{proof}

Let $x$ be a parabolic fixed point of $f$.
Let $q$ be the smallest positive integer so that $(f^q)'(x) = 1$.
Then near $x$, we have 
$$
f^q(z) = (z-x) + a(z-x)^{k+1} + ...
$$
We call $k$ the {\em parabolic index} of $x$ and $k+1$ the {\em multiplicity} of the parabolic fixed point.
Note that $q$ must divide $k$ (see \cite[Chapter 7]{Mil06}).
\begin{lem}\label{lem:ep}
Let $f$ be a rational map with a fat gasket Julia set.
Let $x$ be a common boundary point of two Fatou components $U, V$.
Then $x$ is eventually mapped to a unique parabolic fixed point with parabolic index $2$.
\end{lem}
\begin{proof}
Assume first that $x$ is periodic.
Suppose for contradiction that $x$ is a repelling periodic point.
Then the two Fatou components must touch at an angle at $x$.
This is a contradiction.
Thus $x$ is a parabolic fixed point.

Note that the parabolic index at $x$ is at most $2$, as there are at most 2 attracting petals at $x$.
Suppose for contradiction that the parabolic index is $1$.
Then there is only 1 attracting petal at $x$.
Interchanging the role of $U$ and $V$ if necessary, we assume that $U$ contains the attracting petal of $x$.
Then $\partial U$ has a cusp at $x$, which is a contradiction.
Thus, the parabolic index at $x$ is $2$.

We claim that there is a unique periodic point, which is necessarily a fixed point, on the intersection of the boundaries of two Fatou components.
\begin{proof}[Proof of the claim]
Suppose for contradiction that there are 2 distinct periodic points $x_i \in \partial U_i \cap \partial V_i$, $i=1,2$ where $U_i \neq V_i$ are different Fatou components.
After passing to an iterate, we may assume that $U_i$ and $V_i$ are fixed.
Let $u_i, v_i$ be the corresponding vertices in $\mathcal{G}$.
Since the Fatou graph $\mathcal{G}$ is connected, there is a path containing $u_1, u_2, v_1, v_2$.
Let $\mathcal{C}$ be the shortest path containing $u_1, u_2, v_1, v_2$ with respect to the edge metric on $\mathcal{G}$.
Note that in particular, $\mathcal{C}$ must be a simple path.
Since $u_i, v_i$ are fixed by $f_*$, $f_*(\mathcal{C})$ still contains $u_1, u_2, v_1, v_2$.
Since we assume $\mathcal{C}$ has the shortest length, $f_*(\mathcal{C})$ is a simple path.
By induction, $f_*^n(\mathcal{C})$ is a simple path for any $n$.

Since every vertex is pre-periodic, replace $\mathcal{C}$ by $f_*^k(\mathcal{C})$ for some $k$ if necessary, we assume that all the vertices on $\mathcal{C}$ are fixed.
Since $x_1 \neq x_2$, the length of $\mathcal{C}$ is at least $2$.
Let $a, b, c$ be three consecutive vertices on $\mathcal{C}$, and let $U_a, U_b, U_c$ be the corresponding Fatou components.
Let $x_{ab} \in \partial U_a \cap \partial U_b$ and $x_{bc} \in \partial U_b \cap \partial U_c$.
Then $x_{ab}$ and $x_{bc}$ are both fixed points.
By the previous argument, $x_{ab}$ and $x_{bc}$ are parabolic fixed points with parabolic index $2$.
Thus, $U_b$ contains both the attracting petals for $x_{ab}$ and $x_{bc}$, which is a contradiction.
\end{proof}

By Corollary \ref{cor:is}, every common boundary point of two Fatou components is pre-periodic.
By the previous claim, it is eventually mapped to the unique parabolic fixed point with parabolic index $2$.
\end{proof}

We are now ready to prove Theorem \ref{thm:dfg}.
\begin{proof}[Proof of Theorem \ref{thm:dfg}]
By Corollary \ref{cor:is}, $f$ induces a simplicial map $f_*: \mathcal{G} \longrightarrow \mathcal{G}$.
By Lemma \ref{lem:ep}, every edge of $\mathcal{G}$ is eventually mapped to a fixed edge $E_0$.
\end{proof}

\section{Construction of fat gasket Julia set}\label{sec:rfg}
In this section, we prove the realization theorem of fat gasket when there are no critical points on the Julia set.
We will then consider the quadratic case, and prove Theorem \ref{thm:thmD}.

Let $\mathcal{G}$ be a {\em simple plane graph}, i.e., an embedded simple graph in $\hat\C$.
Two simple plane graphs are said to be isomorphic if there is a homeomorphism of $\hat\C$ inducing the graph isomorphism.
Let $F:\mathcal{G} \longrightarrow \mathcal{G}$ be a map.
We say $F$ is a {\em simplicial branched covering} on $\mathcal{G}$ if 
\begin{itemize}
    \item $F$ is a simplicial map; and
    \item $F$ is the restriction of a branched covering $\widetilde{F}$ on $\hat\C$ with critical points contained in the vertex set of $\mathcal{G}$.
\end{itemize}
We say $F$ has degree $d$ if the branched covering $\widetilde{F}$ has degree $d$.
We apply Thurston theory of rational maps to prove the following key result of this section.
\begin{theorem}\label{thm:nto}
Let $\mathcal{G}$ be a connected simple plane graph.
Let $F: \mathcal{G} \longrightarrow \mathcal{G}$ be a degree $d$ simplicial branched covering.
Let $E_0 = [a,b]$ be an edge of $\mathcal{G}$.
Suppose that
\begin{itemize}
    \item $\mathcal{G}$ is {\em backward invariant};
    \item $E_0$ is fixed and any edge $E$ of $\mathcal{G}$ is eventually mapped to $E_0$;
    \item both $a, b$ are contained in critical cycles;
    \item $\mathcal{G} - \Int(E_0)$ is connected.
\end{itemize}
Then $F: \mathcal{G} \longrightarrow \mathcal{G}$ can be realized as the dynamics on a Fatou graph of a rational map with fat gasket Julia set.
\end{theorem}

\begin{rmk}
We remark that the first three conditions are necessary by Theorem \ref{thm:dfg}.
When there are no critical points on the Julia set, then it is not hard to see that if $\mathcal{G}-\Int(E_0)$ is disconnected, then there exists a simple closed curve $\gamma$ separating the post-critical points in the two components of $\mathcal{G} - \Int(E_0)$, and $\gamma$ is a Levy cycle.
So the topological branched covering is Thurston obstructed.
Therefore, the last condition is also necessary when no critical points are on the Julia set.
\end{rmk}

\subsection{Finite core}\label{subsec:fc}
In this subsection, we shall see that this infinite simple plane graph in Theorem \ref{thm:nto} can be constructed from its finite core.

Let $F: \mathcal{G} \longrightarrow \mathcal{G}$ be a degree $d$ simplicial branched covering as in Theorem \ref{thm:nto}.
Let $\mathcal{H}\subseteq\mathcal{G}$ be a finite connected graph containing all critical values and the fixed edge $E_0$.
Since every edge is eventually mapped to $E_0$, there exists a constant $N$ so that $F^N(\mathcal{H}) = E_0$.

Define $\mathcal{G}^0 = \bigcup_{k=0}^N F^k(\mathcal{H})$ and $\mathcal{G}^1 = F^{-1}(\mathcal{G}^0)$.
We shall call $F: \mathcal{G}^1 \longrightarrow \mathcal{G}^0$ a {\em finite core} of $F: \mathcal{G} \longrightarrow \mathcal{G}$.
We remark that since the subgraph $\mathcal{H}$ in the construction is not unique, finite cores are not unique.

\begin{lem}\label{lem:fc}
Let $F: \mathcal{G} \longrightarrow \mathcal{G}$ be a degree $d$ simplicial branched covering as in Theorem \ref{thm:nto}.
Let $F: \mathcal{G}^1 \longrightarrow \mathcal{G}^0$ be its finite core.
Then
\begin{itemize}
    \item $\mathcal{G}^0 \subseteq \mathcal{G}^1$;
    \item $\mathcal{G}^1$ is connected;
    \item $\mathcal{G} = \bigcup_{k=0}^\infty\mathcal{G}^k$, where $\mathcal{G}^k= F^{-1}(\mathcal{G}^{k-1})$;
    \item $\mathcal{G}^1-\Int(E_0)$ is connected.
\end{itemize}
\end{lem}
\begin{proof}
Since $F(\mathcal{G}^0) \subseteq \mathcal{G}^0$, we have that $\mathcal{G}^0 \subseteq \mathcal{G}^1$.

Since $\mathcal{G}^0$ contains all critical values, $\mathcal{G}^1$ contains all critical points.
Since $\mathcal{G}^1 = \widetilde{F}^{-1}(\mathcal{G}^0)$, $\widetilde{F}$ is a homeomorphism between a face of $\mathcal{G}^1$ and a face of $\mathcal{G}^0$.
Therefore, each face of $\mathcal{G}^1$ is simply connected, so $\mathcal{G}^1$ is connected.

Apply the above argument inductively, we have a sequence of finite connected simple plane graphs
$$
\mathcal{G}^0 \subseteq \mathcal{G}^1 \subseteq \mathcal{G}^2 \subseteq ...
$$
where $\mathcal{G}^k = F^{-1}(\mathcal{G}^{k-1})$.
Since every edge of $\mathcal{G}$ is eventually mapped to $E_0$, we have $\mathcal{G} = \bigcup_{k=0}^\infty\mathcal{G}^k$.

Suppose $\mathcal{G}^1-\Int(E_0)$ is not connected.
Then let $U$ be the unique open face of $\mathcal{G}^1$ whose boundary contains $E_0$.
Note that $U$ has access to two sides of $E_0$.
Let $V = \widetilde{F}(U)$. 
Note that $\partial V$ is a union of edges in $\mathcal{G}^0$.
Since $E_0$ is fixed, and $U$ is a face of $\mathcal{G}^1$, we have $U \subseteq V$.
In particular, $V$ has access to two sides of $E_0$.
Since $\widetilde{F}:U \longrightarrow V$ is a homeomorphism, $\widetilde{F}^{-1}|_{U\longrightarrow V}(U) \subseteq U$ is a face of $\mathcal{G}^{2}$ with access to two sides of $E_0$.
Therefore, $\mathcal{G}^2-E_0$ is not connected.
Thus, inductively, we have $\mathcal{G}^k-\Int(E_0)$ is not connected for all $k$.
This implies $\mathcal{G} - \Int(E_0)$ is not connected, which is a contradiction.
\end{proof}

Conversely, one can easily show that if we start with two finite connected simple plane graphs $\mathcal{G}^0 \subseteq \mathcal{G}^1$, a degree $d$ simplicial branched covering 
$$
F:\mathcal{G}^1 \longrightarrow \mathcal{G}^0
$$
and an edge $E_0 = [a,b]$ of $\mathcal{G}^0$ so that
\begin{itemize}
    \item $E_0$ is fixed and any edge $E$ of $\mathcal{G}^1$ is eventually mapped to $E_0$;
    \item $\mathcal{G}^1 - \Int(E_0) $ is connected.
\end{itemize}
Then we can construct a degree $d$ branched covering $F$ on the union
$$
F: \mathcal{G} = \bigcup_{k=0}^\infty\mathcal{G}^k \longrightarrow \mathcal{G},
$$
where $\mathcal{G}^k = \widetilde{F}^{-1}(\mathcal{G}^{k-1})$ so that it satisfies the assumptions for Theorem \ref{thm:nto}.

\subsection{Thurston theory of rational maps}
In this subsection, we shall briefly summarize some basics of Thurston theory of rational maps.

A post-critically finite branched covering of a topological $2$-sphere $\mathbb{S}^2$ is called a {\em Thurston map}. 
We denote the post-critical set of a Thurston map $f$ by $P(f)$. 
Two Thurston maps $f$ and $g$ are \emph{equivalent} if there exist two orientation-preserving homeomorphisms $h_0,h_1:(\mathbb{S}^2,P(f))\to(\mathbb{S}^2,P(g))$ so that $h_0\circ f = g\circ h_1$ where $h_0$ and $h_1$ are isotopic relative to $P(f)$.

A set of pairwise disjoint, non-isotopic, essential, simple, closed curves $\Sigma$ on $\mathbb{S}^2\setminus P(f)$ is called a {\em curve system}. A curve system $\Sigma$ is called $f$-stable if for every curve $\sigma\in\Sigma$, all the essential components of $f^{-1}(\sigma)$ are homotopic rel $P(f)$ to curves in $\Sigma$. We associate to an $f$-stable curve system $\Sigma$ the {\em Thurston linear transformation} 
$$
f_\Sigma: \R^\Sigma \longrightarrow \R^{\Sigma}
$$
defined as 
$$
f_\Sigma(\sigma) = \sum_{\sigma' \subseteq f^{-1}(\sigma)} \frac{1}{\deg(f: \sigma' \rightarrow \sigma)}[\sigma']_\Sigma,
$$
where $\sigma\in\Sigma$, and $[\sigma']_\Sigma$ denotes the element of $\Sigma$ isotopic to $\sigma'$, if it exists.
The curve system is called {\em irrreducible} if $f_\Sigma$ is irreducible as a linear transformation.
It is said to be a {\em Thurston obstruction} if the spectral radius $\lambda(f_\Sigma) \geq 1$.

We refer the readers to \cite{DH93} for the definition of hyperbolic orbifold, but this is the typical case as any map with more than four postcritical points has hyperbolic orbifold.
Thurston's topological characterization of rational maps says
\begin{theorem}\cite[Theorem 1]{DH93}\label{thm:to}
Let $f$ be a Thurston map which has hyperbolic orbifold. Then $f$ is equivalent to a rational map if and only if $f$ has no Thurston's obstruction.
Moreover, if $f$ is equivalent to a rational map, the map is unique up to M\"obius conjugacy.
\end{theorem}

An arc $\lambda$ in $\mathbb{S}^2\setminus P(f)$ is an embedding of $(0,1)$ in $\mathbb{S}^2\setminus P(f)$ with end-points in $P(f)$.
It is said to be {\em essential} if it is not contractible in $\mathbb{S}^2$ fixing the two end-points.
A set of pairwise non-isotopic essential arcs $\Lambda$ is called an {\em arc system}.
The Thurston linear transformation $f_\Lambda$ is defined in a similar way, and we say that it is irreducible if $f_\Lambda$ is irreducible as a linear transformation.

For a curve system $\Sigma$ (respectively, an arc system $\Lambda$), we set $\widetilde{\Sigma}$ (respectively, $\widetilde{\Lambda}$) as the union of those components of $f^{-1}(\Sigma)$ (respectively, $f^{-1}(\Lambda)$) which are isotopic relative to $P(f)$ to elements of $\Sigma$ (respectively, $\Lambda$).
We will use $\Sigma \cdot \Lambda$ to denote the minimal intersection number between them.
We will be using the following theorem excerpted and paraphrased from \cite[Theorem 3.2]{PT98}.
\begin{theorem}\cite[Theorem 3.2]{PT98}\label{thm:pt}
Let $f$ be a Thurston map, $\Sigma$ an irreducible Thurston obstruction in $(\mathbb{S}^2, P(f))$, and $\Lambda$ an irreducible arc system in $(\mathbb{S}^2, P(f))$.
Assume that $\Sigma$ intersect $\Lambda$ minimally,
then either
\begin{itemize}
\item $\Sigma \cdot \Lambda = 0$; or
\item $\Sigma \cdot \Lambda \neq 0$ and for each $\lambda \in \Lambda$, there is exactly one connected component $\lambda'$ of $f^{-1}(\lambda)$ such that $\lambda' \cap \widetilde{\Sigma} \neq \emptyset$. Moreover, the arc $\lambda'$ is the unique component of $f^{-1}(\lambda)$ which is isotopic to an element of $\Lambda$.
\end{itemize}
\end{theorem}

With this preparation, we are ready to show 
\begin{lem}\label{lem:pcfc}
Let $F: \mathcal{G} \longrightarrow \mathcal{G}$ be a degree $d$ simplicial branched covering satisfying the conditions in Theorem \ref{thm:nto}.
Let $\widetilde{F}$ be the corresponding Thurston map on $\mathbb{S}^2 \cong \hat{\C}$.
Then $\widetilde{F}$ is equivalent to a rational map $f$.

Moreover, the induced dynamics on the Fatou graph of $f$ is conjugate to $F: \mathcal{G} \longrightarrow \mathcal{G}$.
\end{lem}
\begin{proof}
It is easy to verify that $f$ has hyperbolic obifold.
Thus, by Theorem \ref{thm:to}, it is suffices to show there are no Thurston obstructions.

Suppose for constradiction that there is a Thurston obstruction $\Sigma$.
After passing to a subset, if necessary, we may assume that $\Sigma$ is irreducible.

Let $\Lambda = \{E_0\}$. Then $\Lambda$ is an irreducible arc system.
Istoping $\Sigma$, we may assume that $\Sigma$ intersects $\Lambda$ minimally.
Let $\widetilde{\Sigma}_n$ be the union of those components of $f^{-n}(\Sigma)$ which are isotopic to elements of $\Sigma$.

We claim that $\widetilde{\Sigma}_n$ does not intersect $f^{-n}(E_0) - E_0$.
Indeed, applying Theorem \ref{thm:pt} to $f^n$, we are led to the following two cases.
In the first case, $\Sigma \cap E_0 = \emptyset$, so $f^{-n}(\Sigma) \cap f^{-n}(E_0) = \emptyset$ and the claim follows.
In the second case, since $f(E_0) = E_0$, we conclude that $E_0$ is the unique component of $f^{-n}(E_0)$ that is isotopic to $E_0$.
Thus, the only component of $f^{-n}(E_0)$ intersecting $\widetilde{\Sigma}_n$ is $E_0$, and the claim follows.

Since $\mathcal{G}- E_0$ is connected, there exists a finite graph $\mathcal{H} \subseteq \mathcal{G}$ so that $\mathcal{H} - E_0$ is a connected graph containing the post-critical set.
Since every edge is eventually mapped to $E_0$, there exists $N$ so that $\mathcal{H} \subseteq f^{-N}(E_0)$.
Therefore, by the claim, $\widetilde{\Sigma}_N$ does not intersect $\mathcal{H} - E_0$.
Thus, $\Sigma$ does not intersect $\mathcal{H} - E_0$.
This forces $\Sigma$ to be empty, which is a contradiction.

The moreover part follows directly from a standard pull-back argument.
\end{proof}

We are now ready to prove Theorem \ref{thm:nto}.
\begin{proof}[Proof of Theorem \ref{thm:nto}]
By Lemma \ref{lem:pcfc}, there exists a \pcf rational map $f$ whose induced dynamics on the Fatou graph is conjugate to $F: \mathcal{G} \longrightarrow \mathcal{G}$.

To get a fat gasket from the \pcf rational map $f$, we can perform a standard {\em simple pinching} deformation (see \cite{CT18} or \cite{HT04}) to create a double parabolic fixed point at the common boundary of the two fixed Fatou components of $f$.
Let us denote this new map $\tilde{f}$.
Note that $\tilde{f}$ is geometrically finite, and $\tilde{f}: J(\tilde{f}) \longrightarrow J(\tilde{f})$ is topologically conjugate to $f: J(f) \longrightarrow J(f)$.

By the local Fatou coordinate at the double parabolic fixed point, it is easy to see the two fixed Fatou of $\tilde{f}$ are tangent to each other.
By inductively pulling back this tangent point, we see that $J(\tilde{f})$ is a fat gasket. 
Since $\tilde{f}$ and $f$ have the same Julia dynamics, the induced dynamics of $\tilde{f}$ on the Fatou graph is conjugate to $F: \mathcal{G} \longrightarrow \mathcal{G}$.
\end{proof}

\subsection{The quadratic case}
We will now consider the quadratic case and prove Theorem \ref{thm:thmD}.
\begin{proof}[Proof of Theorem \ref{thm:thmD}]
Let $\mathcal{G}$ be the Fatou graph.
By Theorem \ref{thm:dfg}, there exists a unique fixed edge $E_0=[a,b] \subseteq \mathcal{G}$.
Note that $a, b$ are not fixed, as any quadratic rational map with two fixed Fatou component has its Julia set homeomorphic to a circle.
Thus $a, b$ form a periodic 2-cycle under $f_*$.

Note that the Fatou components $U_a \cup U_b$ contain exactly one critical point.
Let $c \notin U_a \cup U_b$ be the other critical point.
We have two cases:

Case 1: If $c$ is on the Julia set, then it is easy to see that the Fatou graph $\mathcal{G}$ is a tree, as $c$ is not on the boundary of a Fatou component.

Case 2: If $c$ is contained in the Fatou set, then $f$ is a geometrically finite rational map.
So there exists a \pcf rational map $g$ with topologically conjugate dynamics on the Julia sets (see \cite[Theorem 1.3]{CT18}).
Note that $g$ is hyperbolic as it is \pcf and there are no critical point on the Julia set.
Since $f$ has a fat gasket Julia set, it has a parabolic fixed point at $\partial U_a \cap \partial U_b$. 
Thus, $f$ is a root of the hyperbolic component containing $g$.
Since every edge is eventually mapped to $E_0$ by Theorem \ref{thm:dfg}, $c$ is eventually mapped to $U_a \cup U_b$.
Therefore, $f$ is a root of a captured type hyperbolic component with an attracting cycle of period $2$.
This proves the first part of Theorem \ref{thm:thmD}.

We now prove the moreover part.
Suppose $f$ is in Case 1 and is geometrically finite.
Let $g$ be the corresponding \pcf rational map.
We can connect the 2 points in the critical 2-cycle of $g$ using two internal rays, and denote this by $I$.
We can choose a small neighborhood $U$ of $I$, so that $\gamma = \partial U$ is a simple closed curve, $g^{-1}(\gamma)$ is isotopic to $\gamma$ relative to the post-critical points of $g$, and $g: g^{-1}(\gamma) \longrightarrow \gamma$ is a two-to-one map.
Thus, $\gamma$ is an equator, and $g$ is a mating of the Basilica with a dendrite polynomial.
Therefore, $f$ is a mating of the fat Basilica with a dendrite polynomial.

Let $g$ be a quadratic \pcf rational map of captured type with a critical 2-cycle.
Let $U_a, U_b$ be the Fatou components for the critical 2-cycle.
Note that the internal fixed rays $\gamma_a, \gamma_b$ of $U_a$ and $U_b$ must land at a common fixed point.
Indeed, otherwise, they land at a 2-cycle $x_a, x_b$, but a quadratic rational map can have at most one 2-cycle, which is a contradiction.
Thus $\partial U_a$ intersects $\partial U_b$.
Let $E_0$ be the union of $\gamma_a$ and $\gamma_b$ together with their common landing point.
Inductively, it is easy to see that 
$\mathcal{G}_n := g^{-n}(E_0)$
is a connected simple graph.
Let $\mathcal{G} = \bigcup_{n=0}^{\infty} \mathcal{G}_n$.
Then it is not hard to see that $g:\mathcal{G}\longrightarrow \mathcal{G}$ satisfies the assumptions in Theorem \ref{thm:nto}.
Indeed, the condition that $\mathcal{G} - \Int(E_0)$ is connected follows as $g$ has no Levy cycle.
Thus, similar to the proof of Theorem \ref{thm:nto}, 
if $f$ is a root of the hyperbolic component containing $g$,
then $f$ has fat gasket.
\end{proof}

\section{Captured type in \texorpdfstring{$\Per_2(0)$}{Per2(0)}}\label{sec:CT}
To understand the moduli space of quadratic rational maps, in \cite{Mil93}, Milnor defines one complex dimensional slices, called $\Per_n(0)$ curves.
A map $f$ is in $\Per_n(0)$ if it has a super-attracting periodic $n$-cycle. 
By Theorem \ref{thm:thmD}, we are interested in maps of captured type in $\Per_2(0)$ (see Figure \ref{fig:Per}).
This space has been studied extensively in the literature.
In particular, it can be interpreted as a partial mating of the filled Julia set of a Basilica and the Mandelbrot set (see\cite{Wit88, Luo95, Dud11}).

In this section, we will study the combinatorics of Fatou graph $\mathcal{G}$ of a \pcf rational map of captured type in $\Per_2(0)$.

\subsection{The canonical finite core and critical loop}
Let $f$ be the \pcf map of a captured type hyperbolic component in $\Per_2(0)$.
Let $\mathcal{G}$ be the Fatou graph of $f$, and $E_0$ be the unique fixed edge of $\mathcal{G}$.

For simplicity, we shall regard $\mathcal{G}$ as a graph embedded in the dynamical plane of $f$, where a vertex is the center of the corresponding Fatou component, and an edge is a union of two internal rays together with their common landing point.
With this identification, we shall also denote the induced map on $\mathcal{G}$ simply by $f: \mathcal{G} \longrightarrow \mathcal{G}$.

\begin{figure}
    \centering
    \includegraphics[width=0.8\textwidth]{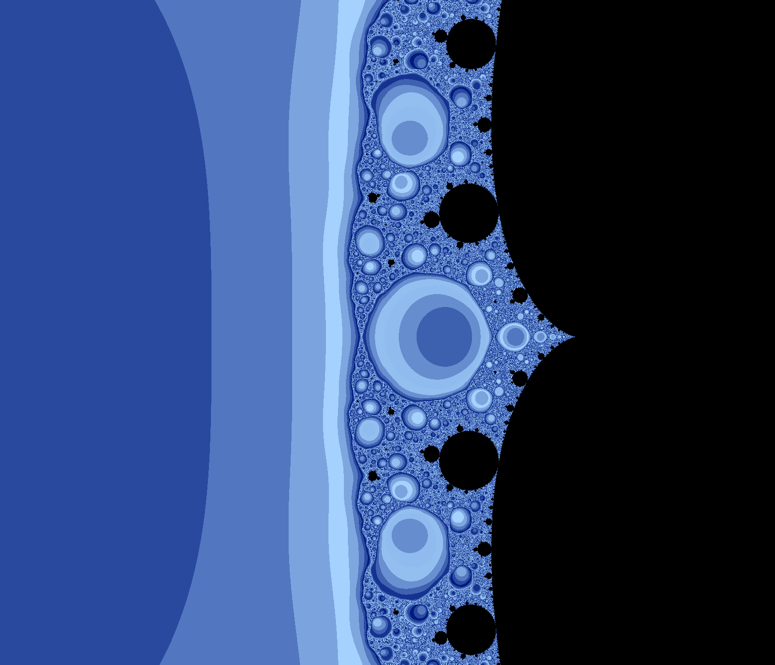}
    \caption{The bifurcation locus of $\Per_2(0)$.}
    \label{fig:Per}
\end{figure}

Let $c \in \mathcal{G}$ be the vertex associated to the strictly pre-periodic critical point.
Let $q$ be the pre-period of $c$. Then $f^q(c) \in \partial E_0$.
We define 
$$
\mathcal{G}^0:= \bigcup_{i=0}^{q-1}f^{-i}(E_0), \text{ and } \; 
\mathcal{G}^1 := f^{-1}(\mathcal{G}^0).
$$
By induction, it is easy to that $\bigcup_{i=0}^{k}f^{-i}(E_0)$ is a tree for all $k \leq q-1$.
Thus, we have
\begin{lem}
The graph $\mathcal{G}^0$ is a tree containing the post-critical set of $\mathcal{G}$.
\end{lem}

Thus, $f: \mathcal{G}^1 \longrightarrow \mathcal{G}^0$ is a finite core of $f: \mathcal{G} \longrightarrow \mathcal{G}$ as introduced in \S \ref{subsec:fc}.
We call this the {\em canonical core} for $f$.

By Lemma \ref{lem:fc}, $\mathcal{G}^1 - \Int(E_0)$ is connected.
Thus, $\mathcal{G}^1$ has at least one simple closed curve that contains $E_0$.
Since $\mathcal{G}^0$ is a tree, it is easy to see that $\mathcal{G}^1$ contains exactly one simple closed curve $\mathcal{C}$.
We call this loop $\mathcal{C}$ the {\em critical loop} for $f$.
The length of the critical loop is an even number $2l$ for some $l \geq 2$.
We call $l$ the {\em critical distance} for $f$.
Since $\mathcal{G}^0$ is a tree, it is easy to show that
\begin{lem}
Let $\mathcal{C}$ be the critical loop. Then $\mathcal{C}$ contains $E_0$ and the critical vertices of $\mathcal{G}$.
Moreover, $f(\mathcal{C})$ is simple path of length $l$, and $f$ maps each of the two components of $\hat{\C} - \mathcal{C}$ homeomorphically to $\hat \C - f(\mathcal{C})$.
\end{lem}

Since $f(E_0) = E_0$, we have $E_0 \subseteq \mathcal{C} \cap f(\mathcal{C})$.
Depending on the intersection $\mathcal{C} \cap f(\mathcal{C})$, we classify \pcf map of captured type in $\Per_2(0)$ into three types (see Figure \ref{fig:TypeI}, \ref{fig:TypeIIA}, \ref{fig:TypeIIB}):
\begin{itemize}
    \item[Type I]: $\mathcal{C} \cap f(\mathcal{C}) = E_0$;
    \item[Type IIA]: $E_0 \subsetneq \mathcal{C} \cap f(\mathcal{C}) \subsetneq f(\mathcal{C})$;
    \item[Type IIB]: $\mathcal{C} \cap f(\mathcal{C}) = f(\mathcal{C})$.
\end{itemize}
We shall call a map $f$ is of Type II if it is either Type IIA or Type IIB.

\subsection{Simple closed curves in $\mathcal{G}$}
In this subsection, we prove that the critical loop has the shortest length among all simple closed curves of $\mathcal{G}$.
\begin{prop}\label{prop:sc}
Let $f$ be a \pcf map of captured type in $\Per_2(0)$.
Let $l$ be the critical distance for $f$.
Then any simple closed curve in $\mathcal{G}$ has length $\geq 2l$.
\end{prop}

The proof of the proposition consists of several lemmas.
\begin{lem}\label{lem:tc}
Let $\mathcal{K} \subseteq \mathcal{G}$ be a simple closed curve.
Then either $\mathcal{K}$ contains both critical points, or the image $f(\mathcal{K})$ contains a simple closed curve of length $\leq$ length of $\mathcal{K}$.
\end{lem}
\begin{proof}
If $\mathcal{K}$ contains at most one critical point, $f$ is locally injective at all but at most one vertex of $\mathcal{K}$. Thus $f(\mathcal{K})$ contains a simple closed curve. It has smaller length as $f$ is a simplicial map.
\end{proof}

\begin{lem}\label{lem:sscc}
Let $k$ be the distance between the two critical points in the graph metric of $\mathcal{G}$.
Then any simple closed curve in $\mathcal{G}$ has length $\geq 2k$.
Moreover, there exists a simple closed curve passing through the two critical points with length $2k$.
\end{lem}
\begin{proof}
Since every edge of $\mathcal{G}$ is eventually mapped to $E_0$, $f^n(\mathcal{K}) = E_0$ for sufficiently large $n$.
Thus, inductively applying Lemma \ref{lem:tc}, and conclude that there exists a simple closed curve containing both critical points with length $\leq$ length of $\mathcal{K}$.
This implies that the length of $\mathcal{K}$ is at least $2k$.

For the moreover part, let $\gamma \subseteq \mathcal{G}$ be a path connecting the two critical points with length $k$.
By the minimality of $\gamma$, it is easy to see that $f(\gamma)$ is a simple path of length $k$.
Let $\mathcal{L} = f^{-1}(f(\gamma))$. Then $\mathcal{L}$ is a simple closed curve of length $2k$.
\end{proof}

\begin{lem}\label{lem:cd}
The critical distance $l$ equals to the distance between the two critical points in the graph metric of $\mathcal{G}$.
\end{lem}
\begin{proof}
Let $k$ be the distance between the two critical points in the graph metric of $\mathcal{G}$, and let $\gamma$ be a path connecting the two critical points with length $k$.
Denote the periodic critical point by $a$ and the strictly pre-periodic critical point by $c$.

Denote the vertices of $\gamma$ in order by $v_0 = a, v_1, ..., v_k = c$.
We claim that for any $i \neq k-1$ and any $j \geq 1$, $f^{j}(v_i) \neq c$.
\begin{proof}[Proof of the claim]
Since $v_k = c$ is strictly pre-periodic, $f^{j}(v_k) \neq c$ for any $j \geq 1$.
Since $f$ is simplicial, we have $d(f^j(a), f^j(v_i)) \leq d(a, v_i) \leq i$ for any $j \geq 1$.
Since $d(a, f^j(a)) \leq 1$, we conclude that for any $i \leq k-2$, we have 
$$
d(a, f^j(v_i)) \leq d(f^j(a), f^j(v_i)) + 1 \leq i+1 \leq k-1.
$$
Since $d(a, c) = k$, this proves the claim.
\end{proof}

If $\gamma \subseteq \mathcal{C}$, then we have $l = k$.

Otherwise, $\gamma$ and $\mathcal{C}$ bounds some simple closed curve $\mathcal{K}$.
Since $\mathcal{C} = f^{-1}(f(\mathcal{C}))$, $f(\mathcal{K})$ contains a simple closed curve.
Thus, inductively applying Lemma \ref{lem:tc}, we can conclude that there exists some vertex $v$ of $\mathcal{K}$ and some iterate $j \geq 1$ so that $f^j(v) = c$.
Since no vertex in $\mathcal{C}$ is mapped to $c$ by $f^j$ for any $j \geq 1$, together with the previous claim, there exists $j$ so that $f^j(v_{k-1}) = c$.
Therefore, $d(f(a), c) = k-1$.

Let $\alpha'$ be a path of length $k-1$ connecting $f(a)$ and $c$, and let $\alpha = \alpha' \cup E_0$.
Note that $\alpha$ is a path of length $k$ connecting the two critical points that contains $E_0$.
Denote the vertices of $\alpha$ in order by $w_0 = a, w_1, ..., w_k = c$.
Since $f(w_1) = a$, we have that for any $i \leq k-1$ and any $j \geq 1$
$$
d(a, f^j(w_i)) \leq d(f^{j}(w_1), f^j(w_i)) + 1 \leq i \leq k-1.
$$
Thus, for any $i \leq k-1$ and any $j \geq 1$, $f^{j}(w_i) \neq c$.
Therefore, the same argument as in the previous paragraph implies that $\alpha \subseteq \mathcal{C}$, so $l = k$.
\end{proof}

\begin{proof}[Proof of Proposition \ref{prop:sc}]
Combining Lemma \ref{lem:sscc} and Lemma \ref{lem:cd}, we have the result.
\end{proof}

\subsection{Shortest anchored simple closed curves}
In this subsection, we further analyze the structures of shortest simple closed curves in $\mathcal{G}$.
More precisely, we say a simple closed curve that contains $E_0$ an {\em anchored simple closed curve}.
We will study these shortest anchored simple closed curves.

\begin{figure}[ht]
  \centering
  \resizebox{0.8\linewidth}{!}{
    \def\svgwidth{\columnwidth}
\begingroup%
  \makeatletter%
  \providecommand\color[2][]{%
    \errmessage{(Inkscape) Color is used for the text in Inkscape, but the package 'color.sty' is not loaded}%
    \renewcommand\color[2][]{}%
  }%
  \providecommand\transparent[1]{%
    \errmessage{(Inkscape) Transparency is used (non-zero) for the text in Inkscape, but the package 'transparent.sty' is not loaded}%
    \renewcommand\transparent[1]{}%
  }%
  \providecommand\rotatebox[2]{#2}%
  \newcommand*\fsize{\dimexpr\f@size pt\relax}%
  \newcommand*\lineheight[1]{\fontsize{\fsize}{#1\fsize}\selectfont}%
  \ifx\svgwidth\undefined%
    \setlength{\unitlength}{581.25bp}%
    \ifx\svgscale\undefined%
      \relax%
    \else%
      \setlength{\unitlength}{\unitlength * \real{\svgscale}}%
    \fi%
  \else%
    \setlength{\unitlength}{\svgwidth}%
  \fi%
  \global\let\svgwidth\undefined%
  \global\let\svgscale\undefined%
  \makeatother%
  \begin{picture}(1,0.85806452)%
    \lineheight{1}%
    \setlength\tabcolsep{0pt}%
    \put(0,0){\includegraphics[width=\unitlength,page=1]{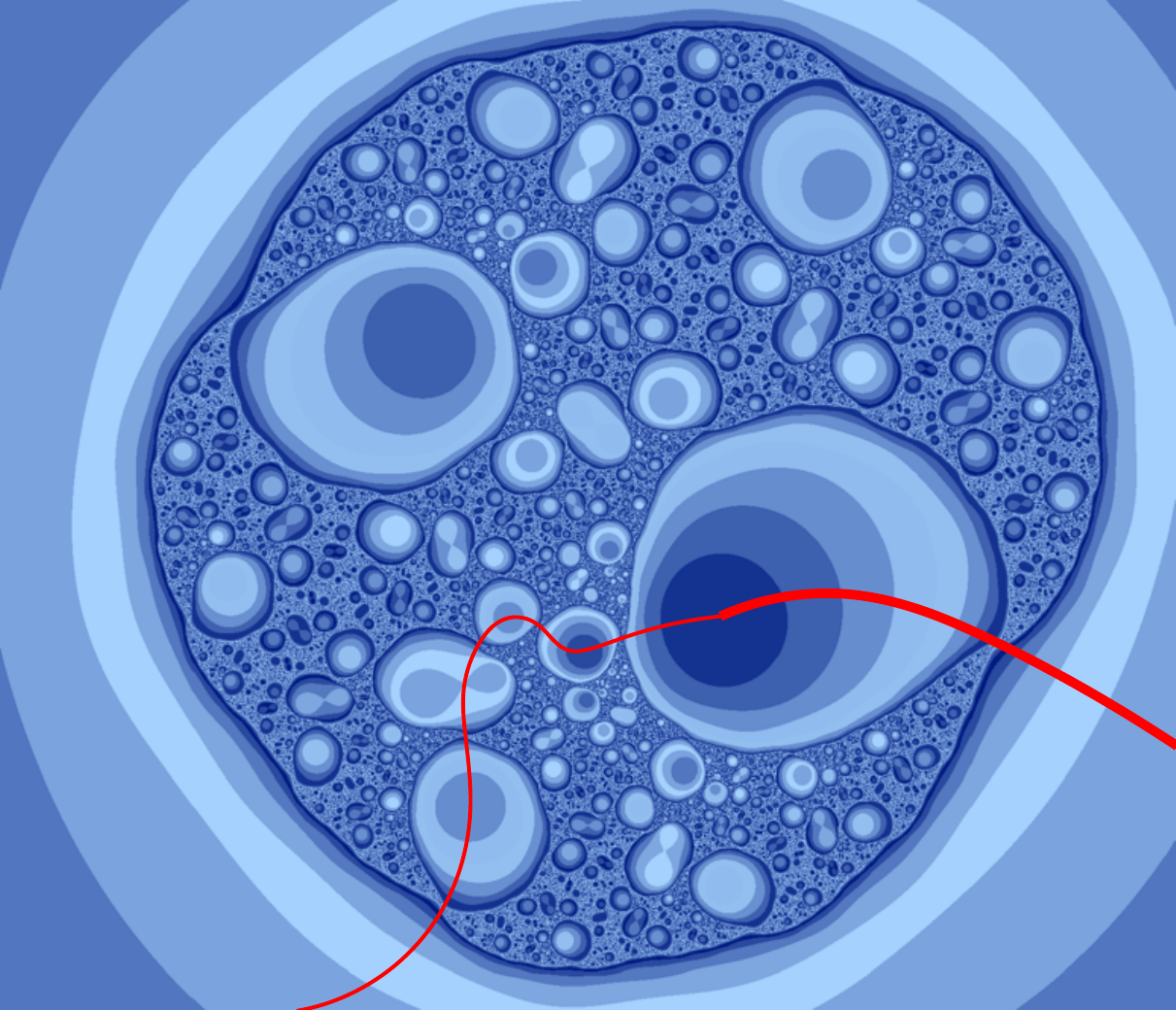}}%
    \put(0.74391901,0.37103839){\color[rgb]{1,0,0}\makebox(0,0)[lt]{\lineheight{1.25}\smash{\begin{tabular}[t]{l}{\Large $E_0$}\end{tabular}}}}%
    \put(0.27605099,0.0618891){\color[rgb]{1,0,0}\makebox(0,0)[lt]{\lineheight{1.25}\smash{\begin{tabular}[t]{l}{\Large $\mathcal{C}$}\end{tabular}}}}%
    \put(0.08112933,0.71913102){\color[rgb]{0,0,0}\makebox(0,0)[lt]{\lineheight{1.25}\smash{\begin{tabular}[t]{l}{\Large $\mathcal{C}_1$}\end{tabular}}}}%
    \put(0.66730697,0.01912464){\color[rgb]{0,0,0}\makebox(0,0)[lt]{\lineheight{1.25}\smash{\begin{tabular}[t]{l}{\Large $\mathcal{C}_2$}\end{tabular}}}}%
    \put(0.7554604,0.79417239){\color[rgb]{0,0,0}\makebox(0,0)[lt]{\lineheight{1.25}\smash{\begin{tabular}[t]{l}{\Large $\mathcal{C}_3$}\end{tabular}}}}%
    \put(0,0){\includegraphics[width=\unitlength,page=2]{TypeI.pdf}}%
  \end{picture}%
\endgroup%

  }
  \caption{The shortest anchored simple closed curves for a Type I map.}
  \label{fig:TypeI}
\end{figure}

\begin{figure}[ht]
  \centering
  \resizebox{0.8\linewidth}{!}{
    \def\svgwidth{\columnwidth}
\begingroup%
  \makeatletter%
  \providecommand\color[2][]{%
    \errmessage{(Inkscape) Color is used for the text in Inkscape, but the package 'color.sty' is not loaded}%
    \renewcommand\color[2][]{}%
  }%
  \providecommand\transparent[1]{%
    \errmessage{(Inkscape) Transparency is used (non-zero) for the text in Inkscape, but the package 'transparent.sty' is not loaded}%
    \renewcommand\transparent[1]{}%
  }%
  \providecommand\rotatebox[2]{#2}%
  \newcommand*\fsize{\dimexpr\f@size pt\relax}%
  \newcommand*\lineheight[1]{\fontsize{\fsize}{#1\fsize}\selectfont}%
  \ifx\svgwidth\undefined%
    \setlength{\unitlength}{581.25bp}%
    \ifx\svgscale\undefined%
      \relax%
    \else%
      \setlength{\unitlength}{\unitlength * \real{\svgscale}}%
    \fi%
  \else%
    \setlength{\unitlength}{\svgwidth}%
  \fi%
  \global\let\svgwidth\undefined%
  \global\let\svgscale\undefined%
  \makeatother%
  \begin{picture}(1,0.85806452)%
    \lineheight{1}%
    \setlength\tabcolsep{0pt}%
    \put(0,0){\includegraphics[width=\unitlength,page=1]{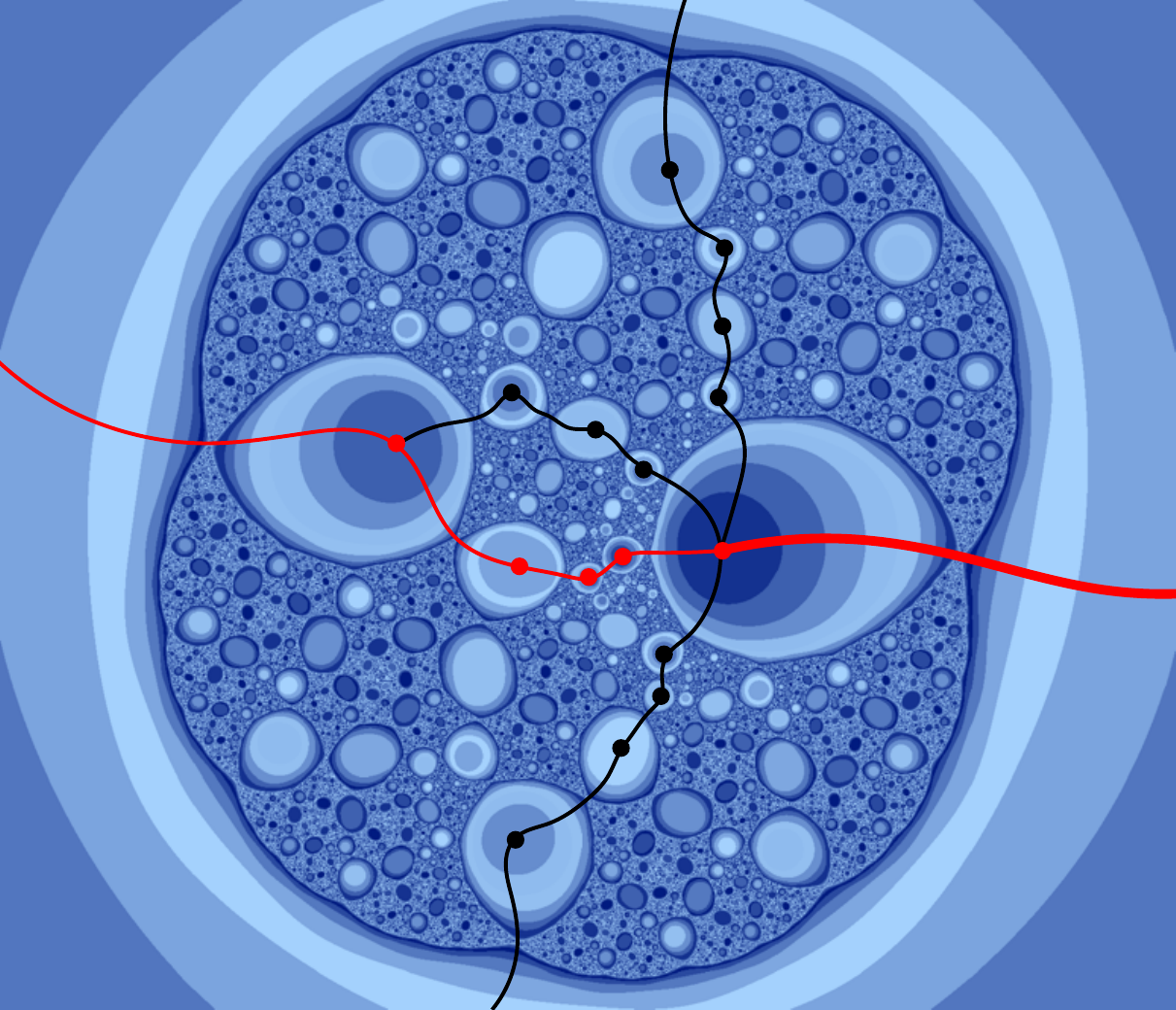}}%
    \put(0.68712205,0.42835205){\color[rgb]{1,0,0}\makebox(0,0)[lt]{\lineheight{1.25}\smash{\begin{tabular}[t]{l}{\Large $E_0$}\end{tabular}}}}%
    \put(0.05698962,0.52911962){\color[rgb]{1,0,0}\makebox(0,0)[lt]{\lineheight{1.25}\smash{\begin{tabular}[t]{l}{\Large $\mathcal{C}$}\end{tabular}}}}%
    \put(0.44301459,0.46044323){\color[rgb]{0,0,0}\makebox(0,0)[lt]{\lineheight{1.25}\smash{\begin{tabular}[t]{l}{\Large $\mathcal{C}_1$}\end{tabular}}}}%
    \put(0.37914137,0.01797135){\color[rgb]{0,0,0}\makebox(0,0)[lt]{\lineheight{1.25}\smash{\begin{tabular}[t]{l}{\Large $\mathcal{C}_2$}\end{tabular}}}}%
    \put(0.58806026,0.82434779){\color[rgb]{0,0,0}\makebox(0,0)[lt]{\lineheight{1.25}\smash{\begin{tabular}[t]{l}{\Large $\mathcal{C}_3$}\end{tabular}}}}%
  \end{picture}%
\endgroup%

  }
  \caption{The shortest anchored simple closed curves for a Type IIA map.}
  \label{fig:TypeIIA}
\end{figure}

\begin{figure}[ht]
  \centering
  \resizebox{0.8\linewidth}{!}{
    \def\svgwidth{\columnwidth}
    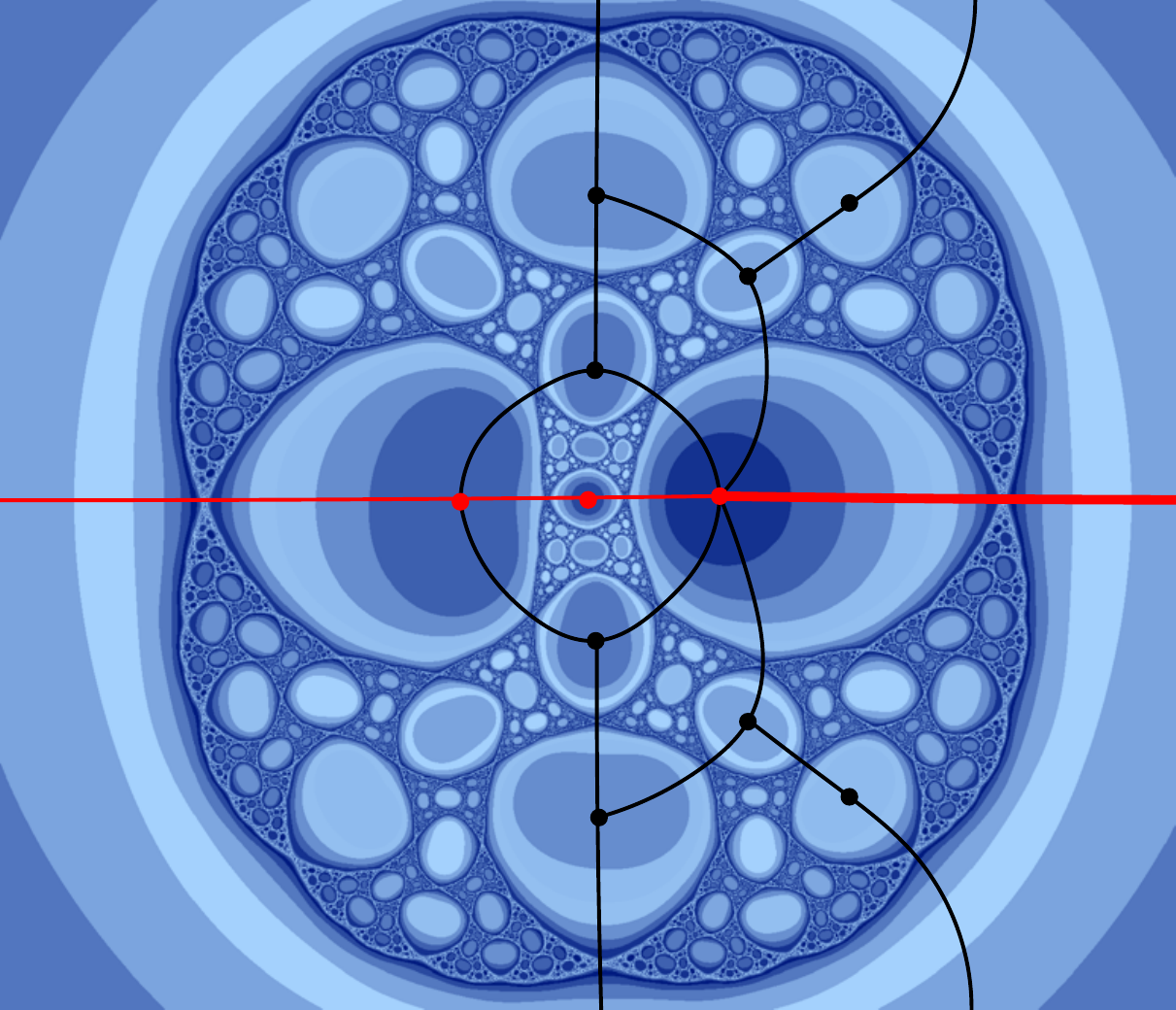

  }
  \caption{The shortest anchored simple closed curves for a Type IIB map.}
  \label{fig:TypeIIB}
\end{figure}

We will use the following lemma
\begin{lem}\label{lem:cl}
Let $\mathcal{K}$ be a shortest anchored simple closed curve.
Then there exists $j\geq 0$ so that $f^j(\mathcal{K})$ is the critical loop $\mathcal{C}$.
\end{lem}
\begin{proof}
We claim that if $\mathcal{K} \neq \mathcal{C}$, then $f(\mathcal{K})$ is a shortest anchored simple closed curve as well.
\begin{proof}[Proof of the claim]
If $\mathcal{K}$ does not contain both critical points, by minimality of the length of $\mathcal{K}$, then $f(\mathcal{K})$ is a shortest anchored simple closed curve.

Otherwise, let $\gamma \subseteq \mathcal{K}$ be a path of length $l$ that connects the two critical points so that $\gamma \nsubseteq \mathcal{C}$.
By minimality and Proposition \ref{prop:sc}, $\gamma$ is disjoint from $\mathcal{C}$ other than the two end points.
Since $\mathcal{K}$ is anchored, we have $E_0 \subseteq \mathcal{K} - \Int(\gamma) \subseteq \mathcal{C}$.
Therefore, $f(\mathcal{K})$ is a simple closed curve. By minimality of the length of $\mathcal{K}$, it is a shortest anchored simple closed curve.
\end{proof}
The lemma follows by inductively applying the above claim.
\end{proof}

Let $\mathcal{C}_1 \neq \mathcal{C}_2$ be two shortest anchored simple closed curves.
We say they are {\em siblings} if $E_0 \subsetneq \mathcal{C}_1 \cap \mathcal{C}_2$, i.e., their intersection contains more than $E_0$.
We say they are {\em non-siblings} if $E_0 = \mathcal{C}_1 \cap \mathcal{C}_2$.
We will prove
\begin{prop}\label{prop:sascc}
Let $f$ be the \pcf map of captured type in $\Per_2(0)$.
Then there are infinitely many shortest anchored simple closed curves, and each is mapped to the critical loop $\mathcal{C}$ by some iterates of $f$.

Moreover,
\begin{itemize}
    \item if $f$ is Type I, then no shortest anchored simple closed curve has siblings;
    \item if $f$ is Type IIA, then the critical loop has one sibling, and other than these two, none has any siblings;
    \item if $f$ is Type IIB, then the critical loop has two siblings, each of which has exactly three siblings, and every other shortest anchored simple closed curve has exactly two siblings.
\end{itemize}
\end{prop}
\begin{proof}
Let $\mathcal{K}$ be a shortest anchored simple closed curve.
By Lemma \ref{lem:cl}, $\mathcal{K}$ is mapped to the critical loop $\mathcal{C}$ by some iterates of $f$.
To prove the there are infinitely many such curves, we will consider three cases.

Case 1: If $f$ is Type I, then $\mathcal{C} \cap f(\mathcal{C}) = E_0$. 
Note that $f$ maps each of the two components of $\hat \C - \mathcal{C}$ homeomorphically to $\hat \C - f(\mathcal{C})$.
Thus, by pulling back, it is easy to verify that $f^{-1}(\mathcal{C})$ is a figure eight curve, and one component, denoted by $\mathcal{C}_1$, is a shortest anchored simple closed curve.
Since $\mathcal{C} \cap f(\mathcal{C}) = E_0$, we have $\mathcal{C} \cap \mathcal{C}_1 = E_0$, so they are not siblings.
Similarly, $f^{-1}(\mathcal{C}_1)$ is a figure eight curve, and contains a shortest anchored simple closed curve, denoted by $\mathcal{C}_1$.
Note that $\mathcal{C}_1$ and $\mathcal{C}_2$ are in two sides of the critical loop $\mathcal{C}$ (see Figure \ref{fig:TypeI}).

Let $U_1$ and $U_2$ be the regions bounded by $\mathcal{C}_1$ and $\mathcal{C}_2$ that do not intersect the critical loop $\mathcal{C}$, and let $U_0$ be the region bounded by $\mathcal{C}$ that does not contain $\mathcal{C}_1$.
Then $f:U_1 \longrightarrow U_0$ and $f: U_2 \longrightarrow U_1$ are homeomorphisms.
Thus, by inductively pulling back, we get a sequence of shortest anchored simple closed curves $\mathcal{C}_n$, where $f(\mathcal{C}_n) = \mathcal{C}_{n-1}$, and none has a sibling.

Case 2: If $f$ is Type IIA, then $E_0 \subsetneq \mathcal{C} \cap f(\mathcal{C}) \subsetneq f(\mathcal{C})$. 
Similar as in Case 1, $f^{-1}(\mathcal{C})$ is a figure eight curve, and contains a shortest anchored simple closed curve, denoted by $\mathcal{C}_1$.
Since $\mathcal{C} \cap f(\mathcal{C})$ contains more than $E_0$, $\mathcal{C}$ and $\mathcal{C}_1$ are siblings (see Figure \ref{fig:TypeIIA}).
Similarly, $f^{-1}(\mathcal{C}_1)$ contains one shortest anchored simple closed curve, denoted by $\mathcal{C}_2$.
Note that $\mathcal{C}_1$ and $\mathcal{C}_2$ are in two sides of the critical loop $\mathcal{C}$. Moreover $\mathcal{C}_1$ and $\mathcal{C}_2$ are not siblings, and neither are $\mathcal{C}$ and $\mathcal{C}_2$.
Then the same pull back argument as in Case 1 gives a sequence of shortest anchored simple closed curves $\mathcal{C}_n$, where $f(\mathcal{C}_n) = \mathcal{C}_{n-1}$, and $\mathcal{C}_n$ has no sibling for all $n \geq 2$.

Case 3: If $f$ is Type IIB, then $\mathcal{C} \cap f(\mathcal{C}) = f(\mathcal{C})$. 
Note that in this case, $\mathcal{C}$ contains both critical values.
Thus, $f^{-1}(\mathcal{C})$ is a union of four simple paths connecting the critical points.
Note that $f$ maps each one of the four components of $\hat \C - f^{-1}(\mathcal{C})$ homeomorphically to a component of $\hat \C - \mathcal{C}$.
It is easy to verify that $f^{-1}(\mathcal{C})$ contains three shortest anchored simple closed curves, one of them is the critical loop.
Denote them by $\mathcal{C}_1, \mathcal{C}_{-1}$ and $\mathcal{C}$.
Therefore, $\mathcal{C}$ has two siblings (see Figure \ref{fig:TypeIIB}).

Let $U_1$ and $U_{-1}$ be the regions bounded by $\mathcal{C}_1$ and $\mathcal{C}_{-1}$ that do not contain the critical loop $\mathcal{C}$, and let $\mathcal{V}_1$ and $\mathcal{V}_{-1}$ be the regions bounded by $\mathcal{C}$ that contains $\mathcal{C}_1$ and $\mathcal{C}_{-1}$ respectively.
Then $f: U_1 \longrightarrow V_{-1}$ and $f: U_{-1} \longrightarrow V_1$ are homeomorphisms.
Thus, by inductively pulling back, we get a sequence of shortest anchored simple closed curves $\mathcal{C}_n, n \in \Z$, where $f(\mathcal{C}_n) = \mathcal{C}_{-\sgn(n)(|n|-1)}$ for $|n| \geq 1$.
Here $\sgn(n)$ represents the sign of $n$, and $\mathcal{C}_0 = \mathcal{C}$.
Moreover, $\mathcal{C}_n$ has two siblings $\mathcal{C}_{n-1}$ and $\mathcal{C}_{n+1}$, and in addition $\mathcal{C}_1$ and $\mathcal{C}_{-1}$ are siblings. There are no other pairs of siblings, which gives the count in the proposition.
\end{proof}

\section{Quasiconformal non-equivalence between quadratic gasket Julia set and limit set}\label{sec:QNE}
In this section, we shall prove Theorem \ref{thm:thmB}.
We shall proceed with proof by contradiction and suppose that there exists a quadratic rational map $f$ whose Julia set is quasiconformally homeomorphic to a geometrically finite gasket limit set.
By Theorem \ref{thm:GFAC}, the contact graph of a geometrically finite gasket limit set is not a tree, and $f$ must have a fat gasket Julia set (see Corollary \ref{cor:GFAC}).
Thus, by Theorem \ref{thm:thmD}, $f$ is a root of a captured type hyperbolic component with an attracting cycle of period $2$.
Since the Julia set of $f$ is homeomorphic to the \pcf center of the corresponding hyperbolic component, Theorem \ref{thm:thmB} follows immediately from the following theorem.

\begin{theorem}
Let $f$ be a \pcf of captured type in $\Per_2(0)$.
Then its Julia set $J$ is not homeomorphic to the limit set of any geometrically finite Kleinian group.
\end{theorem}
\begin{proof}
Suppose for contradiction that $J$ is homeomorphic to the limit set $\Lambda$ of some geometrically finite Kleinian group.
Then the Fatou graph $\mathcal{G}$ of $f$ is homeomorphic to the contact graph of $\Lambda$.

Let $E_0 = [a,b]$ be the unique fixed edge of $\mathcal{G}$ by the induced map of $f$, where $a$ is a critical point of $f$.
By Lemma \ref{lem:pfp}, there exists a subgroup $K \subseteq \Homeo(\mathcal{G})$ isomomorphic to $\Z$ that fixes $E_0$.
Note that any element $g \in K$ must send a shortest anchored simple closed curve to another one.
Note that $g \in K$ preserves sibling-ship.
Thus, by Proposition \ref{prop:sascc}, if $f$ is Type IIA or Type IIB, any element $g \in K$ must fix the critical loop, which is a contradiction.

Therefore, it remains to consider the case that $f$ is Type I.
Note that we can label the shortest anchored simple closed curves by $\mathcal{C}_n, n\geq 0$, where $\mathcal{C}_0 = \mathcal{C}$ is the critical loop, and $f: \mathcal{C}_{i+1} \longrightarrow \mathcal{C}_i$ (see Figure \ref{fig:TypeI}).

Note that $\mathcal{C}_i$ divides the plane into infinitely many crescent shaped region, which we call gaps.
The boundary of any gap contains $\{a,b\} = \partial E_0$.
Denote the region bounded by $\mathcal{C}_0$ and $\mathcal{C}_1$ by $R_0$.
Inductively, we define $R_n, n \in \Z$ as the adjacent gap to $R_{n-1}$ in the counterclockwise direction viewed at the critical point $a$. 

Note that if $n\geq 1$, then $f: R_n \longrightarrow R_{-n}$ is a homeomorphism, and if $n \leq -2$, then $f: R_n \longrightarrow R_{-n-1}$ is a homeomorphism.
Thus, $f$ induces an isomorphism between $\mathcal{G} \cap R_n$ and $\mathcal{G} \cap R_{-n}$ (or $\mathcal{G} \cap R_n$ and $\mathcal{G} \cap R_{-n-1}$ respectively).
Therefore, for $n \geq 2$, $f^2: R_n \longrightarrow R_{n-1}$ is a homeomorphism fixing $a,b$, and for $n \leq -2$, $f^2: R_n \longrightarrow R_{n+1}$ is a homeomorphism fixing $a,b$.
Moreover, $f^2$ induces an isomorphism on the corresponding subgraphs.

Since there exists a subgroup $K \subseteq \Homeo(\mathcal{G})$ isomomorphic to $\Z$ that fixes $E_0$, the subgraph in any gap $R_k$ is isomorphic to $\mathcal{G} \cap R_n$ for some sufficiently large $n$.
By the observation in the previous paragraph, we thus conclude that any the subgraph in any two gaps are homeomorphic, and the homeomorphism can be chosen fixing $a,b$.
In particular, there exist homeomorphisms $g: R_0 \longrightarrow R_1$ and $h: R_{-1} \longrightarrow R_0$ so that
\begin{itemize}
    \item $g(a) = a, g(b) = b$ and $h(a) = a, h(b) = b$;
    \item $g, h$ induce isomorphisms between of the corresponding subgraphs.
\end{itemize}
Therefore, by considering $\tau:= h \circ f \circ g: R_0 \longrightarrow R_0$, we conclude that the subgraph $\mathcal{G} \cap R_0$ is symmetric under an orientation-preserving map that interchanges $a$ and $b$.
The proof of theorem is complete with the following proposition.
\end{proof}

\begin{prop}\label{lem:gap}
Let $f$ be a Type I \pcf of captured type in $\Per_2(0)$.
Then the subgraph $\mathcal{G} \cap R_0$ is not symmetric under an orientation-preserving map $\tau$ that interchanges $a$ and $b$.
\end{prop}

The rest of the section is dedicated to the proof of this proposition.
We shall suppose by contradiction that there exists such a symmetry $\tau$.

We first set up some notations.
Denote the boundary of $E_0$ by $a_0 = a$ and $b_0 = b$, where the critical point corresponds to $a_0$.
Label the vertices on the critical loop $\mathcal{C}_0$ by $a_0, a_1,..., a_{2l-2}, b_0$, and the vertices on the loop $\mathcal{C}_1$ by $b_0, b_1,..., b_{2l-2}, a_0$ (see Figure \ref{fig:R1}).
Note that we have the following dynamics
\begin{itemize}
    \item $f(b_i) = a_i$;
    \item $f(a_1) = a_0 = f(b_0)$;
    \item $f(a_0) = b_0$.
\end{itemize}
We also remark that the two critical points of $f$ are $a_0$ and $a_l$.

\begin{figure}[ht]
  \centering
  \resizebox{0.8\linewidth}{!}{
    \def\svgwidth{\columnwidth}
    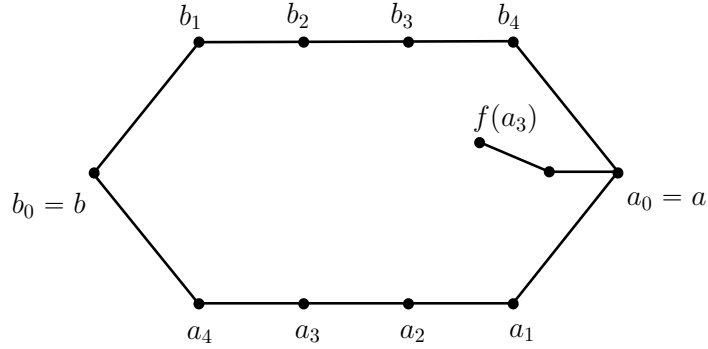

  }
  \caption{An illustration of the gap $R_0$.}
  \label{fig:R1}
\end{figure}

A path $\gamma = [x_0,..., x_{k}]$ in $\mathcal{G}$ is called a {\em local geodesic} if $x_i, x_j$ are not adjacent in $\mathcal{G}$ if $|i-j| \geq 2$.
Note that in particular, if $x, y$ are adjacent in $\mathcal{G}$, then there are no other local geodesics other than the edge $[x,y]$.
We also remark that given any path $\gamma$, one can construct a local geodesic by replacing $[x_i, x_{i+1},..., x_j]$ with $[x_i, x_j]$ if $x_i, x_j$ are adjacent in $\mathcal{G}$.

Since $f$ is post-critically finite, using the expansion of the dynamics away from the post-critical set, it is easy to show that for any pair of vertices $x, y$, there are only finitely many local geodesics of length $k$.

We define an {\em $R_0$-arc} as a local geodesic $\alpha$ that connects two boundary vertices in $\partial R_0$ so that
\begin{itemize}
    \item $\alpha \subseteq \overline{R_0}$;
    \item $\Int \alpha \cap \mathcal{{C}}_0 = \emptyset$.
\end{itemize}

We first prove that
\begin{lem}\label{lem:el}
There exists an $R_0$-arc that connects $a_l$ to $a_0$ so that it has a lift under $f$ that connects $b_l$ and $a_1$.
\end{lem}
\begin{proof}
By pulling back the critical loops using $f$ and taking the associated local geodesic, one can show that there exists some local geodesic $\gamma$ so that
\begin{itemize}
    \item $\Int \gamma \subseteq \Int R_0$;
    \item $\gamma$ connects $a_l$ to $b_i$ for some $i =1,..., 2l-2$;
    \item no interior vertex of $\gamma$ is adjacent to $a_0$.
\end{itemize}
Let $\gamma$ be the one with the shortest length, and suppose that $\gamma$ connects $a_l$ with $b_i$.

We claim that $\Int \gamma \cap f(\mathcal{C}_0) = \emptyset$.
Suppose not. Then we there exists a strictly shorter arc $\delta$ connecting $a_l$ to $v \in f(\mathcal{C}_0)$.
Note that $v \notin \{a, b\}$.
Since $f$ maps the gap $R_0$ homeomorphically to its image, we can find a lift $\delta'$ that connects $b_l$ and $a_j$ for some $j =1,..., 2l-2$.
Then $\tau(\delta')$ is a local geodesic that connects $a_l$ to $b_j$ with $\Int \tau(\delta') \subseteq \Int R_0$ and no interior vertex of $\tau(\delta')$ is adjacent to $a_0$.
This is a contradiction to the minimality of $\gamma$.

It is easy to verify that $R_0$ contains $f(a_l)$.
Let $\alpha = \gamma \cup [b_i, b_{i+1}] \cup ... \cup [b_{2l-2}, a_0]$.
Since no interior vertex of $\gamma$ is adjacent to $a_0$, we have that $\alpha$ is a local geodesic.
Thus, $\alpha$ is an $R_0$-arc that connects $a_l$ and $a_0$.
By the claim, the closed loop
$$
\alpha \cup [a_0, a_1] \cup... \cup [a_{l-1}, a_l]
$$
separates the critical values $f(a_l)$ and $b_0 = f(a_0)$ (see Figure \ref{fig:R1}).
Since the arc $[a_0, a_1] \cup... \cup [a_{l-1}, a_l]$ has a lift connecting $b_l$ and $b_0$, we conclude that the arc $\alpha$ has a lift connecting $b_l$ and $a_1$.
\end{proof}

Let $N$ be the length of the shortest $R_0$-arc that connects $a_l$ to $a_0$ with a lift under $f$ that connects $b_l$ and $a_1$.
Let $A$ be the collection of all $R_0$-arc that connects $a_l$ to $a_0$ of length $\leq N$.
Let $B:= \{\tau(\alpha): \alpha\in A\}$.
Note that $A$ and $B$ are finite sets with the same cardinality.

Let $K$ be the length of the shortest local geodesic $\gamma$ as in the proof of Lemma \ref{lem:el}. Then the proof of Lemma \ref{lem:el} gives that 
\begin{align}\label{eq:bd}
    N \leq K + 2l-1.
\end{align}

We will now prove
\begin{lem}\label{lem:sm}
Let $\beta \in B$. Then $f(\beta) \in A$.
\end{lem}
\begin{proof}
Since $\beta$ is an arc connecting $b_l$ and $b_0$, $f(\beta)$ is an arc connecting $a_l$ and $a_0$.

We will now verify that $f(\beta)$ satisfies the definition of an $R_0$-arc.
Since $\beta = \tau(\alpha)$ for some $\alpha \in A$, $\Int \beta \cap \mathcal{C}_1 = \emptyset$.
Thus, $\Int f(\beta) \cap \mathcal{C}_0 = \emptyset$.

Suppose that $f(\beta) \nsubseteq \overline{R_0}$.
Since $\beta$ connects $b_l$ to $b_0$ in $R_0$, and $f$ sends locally the region bounded between $[b_1, b_0]$ and $[b_0, a_{2l-2}]$ in $R_0$ to the region bounded by $[a_1, a_0]$ and $[a_0, f(a_{2l-2})]$ in $R_0$ (see Figure \ref{fig:R1}), we can decompose $f(\beta)$ as
$f(\beta) = \gamma_1 \cup \gamma_2 \cup \gamma_3$
where $\gamma_1$ connects $a_l$ to $b_i$ for some $i$, $\gamma_2$ connects $b_i$ to $b_j$ for some $j$ and $\gamma_3$ connects $b_j$ to $a_0$.
Moreover, $\Int \gamma_1, \Int \gamma_3 \subseteq \Int R_0$.

We claim that the length $l(\gamma_1) \geq K$.
\begin{proof}[Proof of the claim]
If no interior vertex of $\gamma_1$ is adjacent to $a_0$, then $l(\gamma_1) \geq K$ by the minimality of the definition for $K$.

Otherwise, let $\beta_1 \subseteq \beta$ be the preimage of $\gamma_1$ in $\beta$.
Since $\beta$ is a local geodesic connecting $b_l$ to $b_0$, no interior vertex of $\beta_1$ is adjacent to $b_0$.
Since $f^{-1}(a_0) = \{b_0, a_1\}$, there exists some interior vertex $x\in \beta_1$ that is adjacent to $a_1$.
Consider the truncation $\beta_1' \subseteq \beta_1$ that connects $b_l$ and $x$, and $\widetilde{\beta_1} = \beta_1' \cup [x, a_1]$.
Then $\tau(\widetilde{\beta_1})$ satisfies that
\begin{itemize}
    \item $\Int \tau(\widetilde{\beta_1}) \subseteq \Int R_0$;
    \item $\tau(\widetilde{\beta_1})$ connects $a_l$ to $b_1$;
    \item no interior vertex of $\tau(\widetilde{\beta_1})$ is adjacent to $a_0$.
\end{itemize}
Thus, by minimality of $K$, we have 
$$
K \leq l(\tau(\widetilde{\beta_1})) = l(\widetilde{\beta_1}) = l(\beta_1') + 1 \leq l(\beta_1) = l(\gamma_1).
$$
\end{proof}

By Proposition \ref{prop:sc}, any simple closed curve in $\mathcal{G}$ has length $\geq 2l$.
Thus, for any two vertices $v,w$, there exists at most one path connecting $v, w$ with length $<l$.
Moreover, if there is a path connecting $v, w$ with length $l$, then all the other paths have length $\geq l$.
Since $b_i, b_j, a_0$ all lie on a shortest loop $\mathcal{C}_1$, we have that the lengths $l(\gamma_2), l(\gamma_3) \geq l$.
So 
$$
l(\beta) = l(\gamma_1) + l(\gamma_2) + l(\gamma_3) \geq K+2l > N,
$$ 
which is a contradiction to Equation \ref{eq:bd}.

We now show that $f(\beta)$ is a local geodesic. Suppose not. Let $\delta$ be the associated local geodesic.
Since $\beta$ is a local geodesic and is contained in $\overline{R_0}$, the (non-simple) closed loop by concatenating $f(\beta)$ (from $a_l$ to $a_0$) with $\delta$ (from $a_0$ to $a_l$) separates the two critical values $b_0, f(a_l)$ with winding number $1$.
Thus, $\delta$ is an $R_0$-arc with a lift connecting $b_l$ and $a_1$.
Note that $l(\delta) < l(\beta) \leq N$, which is a contradiction to the minimality of $N$.

Since length $l(f(\beta)) = l(\beta) \leq N$, $f(\beta) \in A$.
\end{proof}

We are ready to prove Proposition \ref{lem:gap}.
\begin{proof}[Proof of Proposition \ref{lem:gap}]
By Lemma \ref{lem:sm}, we have an induced map $f_*:B \longrightarrow A$.
This map is clearly injective as $f$ is injective on $\Int R_0$.
Since $A, B$ have the same cardinality, $f_*$ is also surjective.
On the other hand, by the definition of $N$, there exists $\alpha \in A$ whose lift connects $b_l$ to $a_1$, so this arc $\alpha$ is not in the image of the induced map $f_*$.
This is a contradiction, and the proposition follows.
\end{proof}


\begin{thebibliography}{GLMWY05}

\bibitem[BKK24]{BKK22}
N.~Bogachev, A.~Kolpakov and A.~Kontorovich.
\newblock Kleinian sphere packings, reflection groups, and arithmeticity.
\newblock {\em Math. Comp.}, 93: 505-521, 2024.

\bibitem[BKM09]{BKM09}
M. Bonk, B. Kleiner and S. Merenkov.
\newblock Rigidity of Schottky set.
\newblock {\em Amer. J. Math.}, 131(2): 409--443, 2009.

\bibitem[BLM16]{MLM16}
M. Bonk, M. Lyubich and S. Merenkov.
\newblock Quasisymmetries of Sierpi\'ski carpet Julia sets.
\newblock {\em Adv. Math.}, 301: 383--422, 2016.

\bibitem[BM13]{BM13}
M. Bonk and S. Merenkov.
\newblock Quasisymmetric rigidity of square Sierpin\'ski carpets.
\newblock {\em Ann. of Math. (2)}, 177: 591--643, 2013.

\bibitem[BF11]{BF11}
J. Bourgain and E. Fuchs.
\newblock A proof of the positive density conjecture for integer Apollonian circle packings.
\newblock {\em J. Amer. Math. Soc.}, 24(4): 945--967, 2011.

\bibitem[BK14]{BK14}
J. Bourgain and A. Kontorovich.
\newblock On the local-global conjecture for integral Apollonian gaskets.
\newblock {\em Invent. Math.}, 196(3): 589--650, 2014.

\bibitem[CT18]{CT18}
G. Cui and L. Tan.
\newblock Hyperbolic-parabolic deformations of rational maps.
\newblock {\em Sci. China Math.}, 61: 2157--2220, 2018.

\bibitem[DH93]{DH93}
A. Douady and J. Hubbard.
\newblock A proof of Thurston's topological characterization of rational functions.
\newblock {\em Acta. Math.}, 171: 263--297, 1993. 

\bibitem[Dud11]{Dud11}
D. Dudko.
\newblock Matings with laminations.
\newblock {\em arXiv:1112.4780}, 2011.

\bibitem[GLMWY03]{GLMWY03}
R. Graham, J. Lagarias, C. Mallows, A. Wilks and C. Yan.
\newblock Apollonian circle packings: number theory.
\newblock {\em J. Number Theory}, 100: 1--45, 2003.

\bibitem[GLMWY05]{GLMWY05}
R. Graham, J. Lagarias, C. Mallows, A. Wilks and C. Yan.
\newblock Apollonian circle packings: geometry and group theory I. the Apollonian group.
\newblock {\em Discrete Comput. Geom.}, 34: 547--585, 2005.

\bibitem[HT04]{HT04}
P. Ha{\"i}ssinsky and L. Tan.
\newblock Convergence of pinching deformations and matings of geometrically
  finite polynomials.
\newblock {\em Fund. Math.}, 181:143--188, 2004.

\bibitem[KK23]{KK21}
M.~Kapovich and A.~Kontorovich.
\newblock On superintegral Kleinian sphere packings, gugs, and arithmetic groups.
\newblock {\em J. Reine Angew. Math.}, 798: 105--142, 2023.

\bibitem[KN19]{KN19}
A. Kontorovich and K. Nakamura.
\newblock Geometry and arithmetic of crystallographic sphere packings.
\newblock {\em Proc. Nat. Acad. Sci.}, 116(2): 436--441, 2019.

\bibitem[KO11]{KO11}
A. Kontorovich and H. Oh.
\newblock Apollonian circle packings and closed horospheres on hyperbolic 3-manifolds.
\newblock {\em J. Amer. Math. Soc.}, 24: 603--648, 2011.

\bibitem[LLM22a]{LLM22}
R. Lodge, Y. Luo and S. Mukherjee.
\newblock Circle packings, kissing reflection groups and critically fixed anti-rational maps.
\newblock {\em Forum Math. Sigma}, vol. 10, paper no. e3, 2022.

\bibitem[LLM22b]{LLM22b}
R. Lodge, Y. Luo and S. Mukherjee.
\newblock On deformation space analogies between Kleinian reflection groups and antiholomorphic rational maps.
\newblock {\em Geom. Funct. Anal.}, 32: 1428--1485, 2022.

\bibitem[LLMM23]{LLMM19}
R. Lodge, M. Lyubich, S. Merenkov and S. Mukherjee.
\newblock On dynamical gaskets generated by rational maps, Kleinian groups, and Schwarz reflections.
\newblock {\em Conform. Geom. Dyn.}, 27(1): 1--54, 2023.

\bibitem[Luo95]{Luo95}
J. Luo.
\newblock Combinatorics and holomorphic dynamics: captures, matings and Newton’s method.
\newblock PhD Thesis, Cornell University, 1995.

\bibitem[LZ23]{LZ23}
Y. Luo and Y. Zhang.
\newblock Circle packings, renormalizations and subdivision rules
\newblock{\em arXiv:2308.13151}, 2023.

\bibitem[Mas74]{Mas74}
B. Maskit.
\newblock Intersections of component subgroups of Kleinian groups.
\newblock In L. Greenberg, editor, {\em Discontinuous Groups and Riemann Surfaces: Proceedings of the 1973 Conference at the University of Maryland}, 349--367.
\newblock Ann. Math. St. 79, Princeton University Press, 1974.

\bibitem[McM90]{McM90}
C. McMullen.
\newblock Iteration on Teichm{\"u}ller space.
\newblock {\em Invent. Math.}, 99: 425–454, 1990

\bibitem[Mer14]{Mer14}
S. Merenkov.
\newblock Local rigidity for hyperbolic groups with Sierpi\'nski carpet boundaries.
\newblock {\em Compos. Math.}, 150(11): 1928--1938, 2014.

\bibitem[Mil93]{Mil93}
J. Milnor.
\newblock Geometry and dynamics of quadratic rational maps, with an appendix by the author and Lei Tan.
\newblock {\em Exp. Math.}, 2(1): 37--83, 1993.

\bibitem[Mil06]{Mil06}
J. Milnor.
\newblock Dynamics in One Complex Variable. (AM-160).
\newblock Princeton University Press, 2006.

\bibitem[OS12]{OS12}
H. Oh and N. Shah.
\newblock The asymptotic distribution of circles in the orbits of Kleinian groups.
\newblock {\em Invent. Math.}, 187: 1--35, 2012.

\bibitem[PT98]{PT98}
K. Pilgrim and L. Tan. 
\newblock Combining rational maps and controlling obstructions.
\newblock {\em Ergodic Theory Dynam. Systems}, 18: 221--245, 1998.


\bibitem[QYZ19]{QYZ19}
W. Qiu, F. Yang and J. Zeng.
\newblock Quasisymmetric geometry of Sierpi\'nski carpet Julia sets.
\newblock {\em Fund. Math.}, 244: 73--107, 2019.

\bibitem[Sul81]{Sul81}
D. Sullivan.
\newblock On the ergodic theory at infinity of an arbitrary discrete group of hyperbolic motions.
\newblock In I. Kra and B. Maskit, editors, {\em Riemann Surfaces and Related Topics: Proceedings of the 1978 Stony Brook Conference}, 465-496.
\newblock Ann. Math. St. 97, Princeton University Press, 1981.

\bibitem[Tan92]{Tan92}
L. Tan.
\newblock Matings of quadratic polynomials.
\newblock {\em Ergodic Theory Dynam. Systems}, 12: 589--620, 1992

\bibitem[Thu86]{hyperbolization1}
W.  Thurston.
\newblock Hyperbolic structures on 3-manifolds I: Deformation of acylindrical manifolds.
\newblock {\em Ann. of Math.}, 124: 203--246, 1986.

\bibitem[Wit88]{Wit88}
B. Wittner.
\newblock On the bifurcation loci of rational maps of degree two.
\newblock PhD Thesis, Cornell University, 1988.

\bibitem[Zha22]{Zha22}
Y. Zhang.
\newblock Elementary planes in the Apollonian orbifold.
\newblock {\em Trans. Amer. Math. Soc.}, 376(1): 453--506, 2023.

\end{thebibliography}
\end{document}